\documentclass[13pt,a4paper,oneside]{amsart}
\usepackage{mathrsfs}
\usepackage{amsfonts}
\usepackage{fullpage}
\usepackage{mathrsfs,amsmath,amsfonts,amssymb,amsbsy,amscd}
\usepackage{CJK,caption2,cite,color, latexsym}
\usepackage{xy,graphicx}
\usepackage{fancyhdr}
\usepackage[colorlinks=true,citecolor=blue,linkcolor=red]{hyperref}
\makeatletter
\@namedef{subjclassname@2020}{\textup{2020} Mathematics Subject Classification}
\makeatother
\xyoption{all}
\allowdisplaybreaks

\fontsize{9pt}{9pt}\selectfont

\numberwithin{equation}{section}
\swapnumbers
\newtheorem{theorem}[equation]{Theorem}
\newtheorem{lemma}[equation]{Lemma}
\newtheorem{proposition}[equation]{Proposition}
\newtheorem{corollary}[equation]{Corollary}
\theoremstyle{definition}
\newtheorem{definition}[equation]{Definition}
\newtheorem{example}[equation]{Example}
\theoremstyle{remark}
\newtheorem{remark}[equation]{Remark}
\def\bga{\boldsymbol{\alpha}}
\def\bgb{\boldsymbol{\beta}}
\def\bgl{\boldsymbol{\lambda}}
\def\bkp{\boldsymbol{\kappa}}
\def\bmu{\boldsymbol{\mu}}
\def\bnu{\boldsymbol{\nu}}
\def\btau{\boldsymbol{\tau}}
\def\bl{\boldsymbol{\ell}}
\def\bk{\mathbf{k}}
\def\bx{\mathbf{x}}
\def\by{\mathbf{y}}
\def\gl{\lambda}
\def\bz{\mathbb{Z}}

\def\fa{\mathfrak{a}}
\def\fb{\mathfrak{b}}
\def\fc{\mathfrak{c}}
\def\fs{\mathfrak{s}}
\def\ft{\mathfrak{t}}
\def\fu{\mathfrak{u}}
\def\fv{\mathfrak{v}}
\def\End{\mathrm{End}}
\def\std{\mathrm{Std}}

\def\mpn{\mathscr{P}_{m,n}}
\newdimen\hoogte    \hoogte=16pt    
\newdimen\breedte   \breedte=16pt   
\newdimen\dikte     \dikte=1.0pt    

\newenvironment{point}[2]%
  {\vspace{0.5\jot}\ifx*#2\let\pointlabel\relax\else\def\pointlabel{#2}\fi
   \refstepcounter{equation}\trivlist
   \item[\hskip\labelsep\theequation.
         \ifx\pointlabel\relax\else\space\pointlabel\space\fi]
   \ignorespaces #1
  \vspace{0.5\jot}}{\relax}

\newenvironment{Young}{\begingroup
       \def\vr{\vrule height0.6\hoogte width\dikte depth 0.2\hoogte}
       \def\fbox##1{\vbox{\offinterlineskip
                    \hrule height0.6\dikte
                    \hbox to \breedte{\vr\hfill$##1$\hfill\vr}
                    \hrule height0.6\dikte}}
       \vtop\bgroup \offinterlineskip \tabskip=-\dikte \lineskip=-\dikte
            \halign\bgroup &\fbox{##\unskip}\unskip  \crcr}
     {\egroup\egroup\endgroup}

\def\diagram(#1){\begin{Young}#1\cr\end{Young}}

\def\Tritab(#1|#2|#3){\begin{Young}#1\cr\end{Young}\,;\,\
                              \begin{Young}#2\cr\end{Young}\,;\,\
                              \begin{Young}#3\cr\end{Young}}
\def\Twotab(#1|#2){\begin{Young}#1\cr\end{Young}\,;\,
\begin{Young}#2\cr\end{Young}}

\newenvironment{colorYoung}{\begingroup
       \def\vr{\vrule height0.8\hoogte width\dikte depth 0.2\hoogte}
       \def\fbox##1{\vbox{\offinterlineskip
                    \hrule height1.0\dikte
                    \hbox to 1.4\breedte{\vr\hfill$##1$\hfill\vr}
                    \hrule height1.2\dikte}}
       \vtop\bgroup \offinterlineskip \tabskip=-\dikte \lineskip=-\dikte
            \halign\bgroup &\fbox{##\unskip}\unskip  \crcr}
     {\egroup\egroup\endgroup}

\def\Tricolor(#1|#2|#3){\begin{colorYoung}#1\cr\end{colorYoung}\,;\,\
                              \begin{colorYoung}#2\cr\end{colorYoung}\,;\,\
                              \begin{colorYoung}#3\cr\end{colorYoung}}
 \def\Twocolor(#1|#2){\begin{colorYoung}#1\cr\end{colorYoung}\,;\,\
                              \begin{colorYoung}#2\cr\end{colorYoung}}
 \def\colordiagram(#1){\begin{colorYoung}#1\cr\end{colorYoung}}
\makeatother
\begin{document}
\setlength{\itemsep}{0.25cm}
\fontsize{13}{\baselineskip}\selectfont
\setlength{\parskip}{0.4\baselineskip}
\title[Cyclotomic Schur superalgebras]{\fontsize{9}{\baselineskip}\selectfont Cyclotomic $q$-Schur superalgebras}
\author{Deke Zhao}

\address{\bigskip\hfil\begin{tabular}{l@{}}
          School of Applied Mathematics\\
            Beijing Normal University at Zhuhai, Zhuhai, 519087\\
             China\\
             E-mail: \textit{deke@amss.ac.cn} \hfill
          \end{tabular}}

\dedicatory{Dedicated to Professor Yingbo Zhang on the occasion of her 75th birthday}
\vspace*{-3mm}
\begin{abstract}The paper aims to introduce the cyclotomic $q$-Schur superalgebra via the permutation supermodules of the cyclotomic Hekce algebra and investigate its structure. In particular, we show that the cyclotomic $q$-Schur superalgebra is a cellular superalgebra and establish the double centralizer property between the cyclotomic Hecke algebra and the cyclotomic $q$-Schur superalgebra.
\end{abstract}
\maketitle
\section{Introduction}
In \cite{DJM}, Dipper, James and Mathas introduced the cyclotomic $q$-Schur algebra for the Ariki-Koike algebra or so-called the cyclotomic Hecke algebra associated to the complex reflection group of type $G(m,1,n)$, which is a natural generalization of Dipper and James' $q$-Schur algebra \cite{DJ89} for the Iwhaori-Hecke algebra of type $A$.
They showed the cyclotomic $q$-Schur algebra is a cellular algebra in the sense of Graham and Lehrer \cite{GL} and Mathas \cite[Theorem~5.3]{Mathas-ASPM} showed the Schur-Weyl duality (i.e., the double centralizer property) holds between the cyclotomic Hecke algebra and the cyclotomic $q$-Schur algebra under certain conditions.

Motivated by the quantum Schur-Weyl duality between quantum superalgebra and Iwahori-Hecke algebra of type $A$ proved in \cite{Moon,M}, Du and Rui \cite{DR} introduced the quantum Schur superalgebra for the Iwahori-Hecke algebra of type $A$ and showed it is a cellular (super)algebra. In \cite{Zhao-duality}, the author establishes the super Schur-Weyl reciprocity between quantum superalgebra and cyclotomic Hecke algebra , which is a super analogue of the Schur-Weyl reciprocity proved independently by Sakamoto and Shoji \cite{SS} and Hu \cite{Hu}.

Inspired by the aforementioned works, the aim of this paper is to  introduce  the cyclotomic $q$-Schur superalgebra and study its properties along the line of \cite{DJM}. More precisely, we first define the ``permutation supermodule" of the cyclotomic Hecke algebra (see Definition~\ref{Def:Supermodule}) and describe its basis explicitely (see Theorem~\ref{Them:M-basis}). Then the cyclotomic $q$-Schur superalgebra is defined as the endomorphism (super)algebra of the direct sum of all ``permutation supermodules" (see Definition~\ref{Def:Cyclotomi-Schur}), which is a natural generalization of  Dipper-James-Mathas' cyclotomic $q$-Schur algebra and Du-Rui's  quantum Schur superalgebra.

 Further we show that the cyclotomic $q$-Schur superalgebra enjoies many properties in common with the cyclotomic $q$-Schur algebra and the quantum Schur superalgebra.  In particular, we show that the cyclotomic $q$-Schur superalgebra is a cellular (super)algebras (see Theorem~\ref{Them:Cellular-algebra}), which enables us to get the classification of (ordinary) irreducible representations and prove that the cyclotomic $q$-Schur superalgebra is  quasi-hereditary by applying Graham and Lehrer's theory on cellular algebras. We also establish the double centralizer property between the cyclotomic Hecke algebra  and the cyclotomic $q$-Schur superalgebra (see Theorem~\ref{Them:Schur-Weyl}).

Let us remark that there are various cyclotomic $q$-Schur algebras: Ariki \cite{A} introduced a cyclotomic $q$-Schur algebra as the endomorphism algebra of the action of cyclotomic Hecke algebra on a $q$-tensor space and showed that this cyclotomic $q$-Schur algebra is a quotient of a quantum algebra; Lin and Rui \cite{LR} defined a special cyclotomic $q$-Schur algebra via the affine tensor space and showed it is a quotient of an affine quantum group; Sawada and Shoji \cite{SS-Schur} defined the cyclotomic $q$-Schur algebra as the endomorphism algebra of the action of the modified Ariki-Koike algebra on a $q$-tensor space; Deng et.\,al.\,\cite{DDY} introduced the slim cyclotomic $q$-Schur algebra and constructed its basis via the symmetric polynomials in Jucys-Murphy elements of the cyclotomic Hecke algebra.  It may be interesting to formulate the super-versions of the aforementioned cyclotomic $q$-Schur algebras.

This paper is organized as follows. After a brief review of the definition of the cyclotomic Hecke algebra and its  cellular basis in Section~\ref{Sec:Cyclotomic-Hecke-Algebras}, we begin to fix the combinatoric notations and  give some useful facts on semistandard tableaux in Section~\ref{Sec:super-tableaux}. The permutation supermodule of the cyclotomic Hecke algebra is defined and its basis is described explicitly in Section~\ref{Sec:Permutation-supermodules}.  Section~\ref{Sec:Cyclotomic-Schur-superalgebras} devotes to introduce the cyclotomic $q$-Schur superalgebra and to show its cellularity. We study the representations of the cyclotomic $q$-Schur superalgebra via its cellularity  and  establish the Schur-Weyl duality between the cyclotomic Hecke algebra and the cyclotomic  $q$-Schur superalgebra in the last section.

 Throughout the paper, $m,n$ are fixed positive integers not less than one;  $k,\ell$ are non-negative integers with $k+\ell>0$, $k_1, \ldots, k_m$, $\ell_1, \ldots, \ell_m$ are non-negative integers such that  $\sum_{i=1}^mk_i=k$ and $\sum_{i=1}^m\ell_i=\ell$. We denote by  $\bk=(k_1, \ldots, k_m)$, $\bl=(\ell_1, \ldots, \ell_m)$ and write $\bk|\bl=(k_1|\ell_1, \ldots, k_m|\ell_m)$.

\noindent\textbf{Acknowledgements.} Part of this work was carried out while the author was visiting the Chern Institute of Mathematics (CIM) in Nankai University during Summer of 2019 and he wish to thank the institute for the hospitality during his visit.  The author is very grateful to the anonymous referee  for her/his  numerous useful comments and corrections. This work was supported the National Natural Science Foundation of China (Grant No. 11871107).
\section{Cyclotomic Hecke algebras}\label{Sec:Cyclotomic-Hecke-Algebras}

In this section, we recall the definition of the cyclotomic Hecke algebra and review briefly its cellular basis and some related facts. The main references for this section are \cite{AK} and \cite{DJM}.

\begin{definition}[\cite{AK,BM:cyc}]\label{Def:CHA}Let $R$ be an integral domain with unitary and let $q, Q_1, \ldots, Q_m$ be elements of $R$ with $q$ invertible. The \textit{cyclotomic Hecke algebra} $H$ is the unital associative $R$-algebra  generated by
$T_0,T_1,\dots,T_{n-1}$ and subject to relations
\begin{align*}&(T_0-Q_1)\dots(T_0-Q_m)=0,&&\\
&T_0T_1T_0T_1=T_1T_0T_1T_0,&&\\
&T_i^2=(q-1)T_i+q &&\text{ for }1\leq i<n,\\
&T_iT_j=T_jT_i &&\text{ for }|i-j|>2,\\
 &T_iT_{i+1}T_i=T_{i+1}T_{i}T_{i+1} &&\text{ for }1\leq i<n-1.\end{align*}
\end{definition}

Remark that if  $\xi\in R$ is a primitive $m$-th root of $1$,  $q=1$, and $Q_i=\xi^i$ for $i=1, \ldots, m$. Then $H$ is isomorphic to the group algebra $R\mathfrak{S}_n \ltimes(\mathbb{Z}/m\mathbb{Z})^n$ of the wreath product of the cyclic group $\mathbb{Z}/m\mathbb{Z}$ with the symmetric group $\mathfrak{S}_n$ on $n$ letters.

For a finite set $X$, let $\mathfrak{S}_X$be the group of all permutations of $X$. Thus $\mathfrak{S}_n=\mathfrak{S}_{\{1,2,\ldots, n\}}$ with simple transposes $s_i=(i,i+1)$, $i=1, \ldots, n-1$. Let  $s_{i_1}s_{i_2}\cdots s_{i_k}$ be a reduced expression for $w\in \mathfrak{S}_n$. Then $T_{w}:=T_{i_1}T_{i_2}\cdots T_{i_k}$ is independent of the choice of reduced expression of $w$ and  $\{T_{w}|w\in \mathfrak{S}_n\}$ is linear basis of the Iwahori-Hecke algebra associated to $\mathfrak{S}_n$, which is the subalgebra of $H$ generated by $T_1, \ldots, T_{n-1}$.

 Following \cite{DJM}, we define inductively the \textit{Jucys-Murphy elements} of $H$ as following
\begin{equation*}\label{Equ:JM-def}
J_1\!:=T_0 \quad \text{ and }\quad J_{i+1}\!:=q^{-1}T_{i}J_iT_{i}, \quad i=1, \cdots, n-1,
  \end{equation*}
that is, $J_i=q^{1-i}T_{i-1}\ldots T_1T_0T_1\ldots T_{i-1}$ for $i=1, \ldots, n$.

It is easy to show the following facts.

\begin{lemma}[\protect{\cite[(2.1)]{DJM}}]\label{Lemm:J-property}Assume that $1\le i<n$ and $1\le j\le n$. Then \begin{enumerate}\setlength{\itemsep}{1\jot}
\item  $J_iJ_j=J_jJ_i$.
\item $T_iJ_j=J_jT_i$ if $i\neq j-1, j$.
 \item $T_i(J_iJ_{i+1})=(J_iJ_{i+1})T_i$ and $T_i(J_i+J_{i+1})=(J_i+J_{i+1})T_i$.
   \item If $a\in R$ and $i\neq j$ then $T_i$ commutes with $(J_1-a)(J_2-a)\cdots(J_j-a)$.
   \end{enumerate}
\end{lemma}

The following fact gives a basis for $H$.

\begin{theorem}[\protect{\cite[Theorem~3.10]{AK}}]\label{Them:Basis} The algebra $H$ is a free $R$-module with basis
\begin{equation*}
  \{J_1^{c_1}J_2^{c_2}\cdots J_{n}^{c_n}T_{w}|w\in \mathfrak{S}_n, 0\leq c_i<m \text{ for }i=1,\ldots, n\}.
\end{equation*}
  In particular, $H$ is free $R$-module of rank $H=m^nn!$.
\end{theorem}

Now let $*$ be the anti-automorphism of $H$ determined by $T_i^*=T_i$ for $i=0,1,\ldots, n-1$. Then $T_w^*=T_{w^{-1}}$ for  $w\in \mathfrak{S}_n$ and $J_i^*=J_i$ for $i=1, 2, \ldots,n$.  Thus Theorem~\ref{Them:Basis} shows that $H$ is a free $R$-module with basis
\begin{equation}\label{Equ:Standard-basis}
  \{T_{w}J_1^{c_1}J_2^{c_2}\cdots J_{n}^{c_n}|w\in \mathfrak{S}_n, 0\leq c_i<m \text{ for }i=1,\ldots, n\}.
\end{equation}
Then for any right $H$-module $M$, the anti-automorphism $*$ enables us to yield a left $H$-module $M^*$ such that $M^*=M$ as free $R$-modules and $hm:=mh^*$ for $h\in H,m\in M^*$.

\begin{point}{}*
Recall that a \textit{composition}  (resp. \textit{partition}) $\gl=(\gl_1, \gl_2, \ldots)$ of $n$, denote $\gl\models n$ (resp. $\gl\vdash n$), is a sequence (resp. weakly decreasing sequence) of  non-negative integers such that $|\gl|=\sum_{i\geq1}\gl_i=n$ and we write $\ell(\gl)$ the length of $\gl$, i.e., the number of nonzero parts of $\gl$.
The {\it Young diagram}  of a composition $\lambda=(\lambda_1, \lambda_2,\ldots)$ is the set
\begin{equation*}
 \lambda:= \{(i,j)|1\leq i\text{ and } 1\leq j\leq \lambda_i\},
\end{equation*}
which regards as a collection of boxes arranged in left-justified row with $\lambda_i$ boxes in the $i$-th row. The conjugate of a partition $\gl$ is the partition $\gl^{'}$ whose Young diagram is the transpose of the Young diagram of $\gl$.

 A \textit{multicomposition} (resp. {\it multipartition}) of $n$ is an ordered $m$-tuple $\bgl=(\gl^{(1)}; \ldots; \gl^{(m)})$ of compositions (resp. partitions) $\lambda^{(i)}$ such that
$n=\sum_{i=1}^m|\lambda^{i}|$. We call $\gl^{(i)}$ the $i$-th component of $\bgl$ and denote by  $\mathscr{C}_{m,n}$ (resp. $\mpn$)  the set of all multicompositions (resp. multipartitions) of $n$. Then $\mpn$ is a poset under dominance $\unrhd$, where
\begin{equation*}
\bgl\unrhd\bmu\Longleftrightarrow\displaystyle\sum_{a=1}^{i-1}|\gl^{(a)}|+\sum_{b=1}^j\gl_b^{(i)}\geq
\sum_{a=1}^{i-1}|\mu^{(a)}|+\sum_{b=1}^j\mu_b^{(i)}\quad
\text{ for all }1\leq i\leq m \text{ and }j\geq 1.\end{equation*}
 We write $\bgl\rhd\bmu$
if $\bgl\unrhd\bmu$ and $\bgl\neq \bmu$.

The \textit{Young diagram} of
a multicomposition $\bgl$ is the ordered $m$-tuple of the diagrams of its components, or equivalently, the set
\begin{equation*}
  \bgl:=\{(i,j,c)\in\bz_{>0}\times\bz_{>0}\times \{1, \dots, m\}|1\le j\le\lambda^{(c)}_i\}.
\end{equation*}
We may and will identify $\bgl$ with its Young diagram.
\end{point}

\begin{point}{}*
For $\bgl\in\mpn$, a {\it $\bgl$-tableau} $\ft=(\ft^{(1)}; \dots; \ft^{(m)})$  is obtained  from $\bgl$  by inserting the number $1, 2, \ldots, n$ into its boxes. We may and will identify a tableau $\ft$ with an $m$-tuple of tableaux
$\ft=(\ft^{(1)}; \dots; \ft^{(m)})$, where  the {\it$c$-component} $\ft^{(c)}$ of $\ft$  is a $\lambda^{(c)}$-tableau for $c=1, \cdots, m$.  We  write $\text{Shape}(\ft)=\bgl$ if $\ft$ is a $\bgl$-tableau.

 A $\bgl$-tableau $\ft=(\ft^{(1)}; \dots; \ft^{(m)})$ is {\it row} (resp. {\it column}) {\it standard} if the entries of $\ft^{(c)}$ are strictly increasing in each row (resp. column)  for $c=1, \ldots, m$.  A $\bgl$-tableau is {\it  standard} if it is both row standard and column standard and we denote by $\std(\bgl)$ the set of all standard $\bgl$-tableaux. For example,  let $\ft^{\bgl}$ (resp. $\ft_{\bgl}$) be the
$\bgl$-tableau with the numbers $1,2,\dots,n$ entered in
order first along the rows (resp. columns) of the first component, and then along the rows (resp. columns) of the second component, and so on. Then $\ft^{\bgl}$ and $\ft_{\bgl}$ are standard $\bgl$-tableaux.

For a composition $\lambda^{(i)}=(\lambda_1^{(i)}, \lambda^{(i)}_2,\ldots)$,  the Young subgroup associated to $\lambda^{(i)}$ is the subgroup   \begin{equation*}\mathfrak{S}_{\lambda^{(i)}}=\mathfrak{S}_{\{1, \ldots, \lambda_1^{(i)}\}}\times \mathfrak{S}_{\{\lambda_1^{(i)}+1, \ldots, \lambda^{(i)}_2\}}\times\cdots\end{equation*} of $\mathfrak{S}_{\{1, \ldots, |\lambda^{(i)}|\}}$. Similarly, the Young subgroup associated to  $\bgl=(\gl^{(1)};\ldots;\gl^{(m)})\in\mathscr{C}_{m,n}$ is the subgroup   $\mathfrak{S}_{\bgl}=\mathfrak{S}_{\gl^{(1)}}\times\mathfrak{S}_{\gl^{(2)}}\times \cdots\times \mathfrak{S}_{\gl^{(m)}}$ of $\mathfrak{S}_n$.

Now the symmetric  group $\mathfrak{S}_n$ acts from the right on the set of $\bgl$-tableaux by permuting its entries in each tableau.  Clearly the row stabilizer of $\ft^{\bgl}$ is the Young subgroup of $\mathfrak{S}_{\bgl}$ of $\mathfrak{S}_n$. For a row standard $\bgl$-tableau $\fs$, let $d(\fs)$ be the element of $\mathfrak{S}_n$ such that $\fs=\ft^{\bgl}d(\fs)$. Then $d(\fs)$ is a distinguished  right coset representative of $\mathfrak{S}_{\bgl}$ in $\mathfrak{S}_n$ and there is a correspondence between the set of row standard $\bgl$-tableaux and the set of right coset representatives of $\mathfrak{S}_{\bgl}$ in $\mathfrak{S}_n$.\end{point}

\begin{example}Let $\bgl=((3,2);(1^2);(2,1))$. Then the Young diagram of $\bgl$ is
 \begin{eqnarray*}&&
  \protect{\bgl=\left(\,\Tritab(&&\cr&|\cr|&\cr)\right)}.\end{eqnarray*}
  The standard tableaux $\ft^{\bgl}$ and $\ft_{\bgl}$ are
   \begin{eqnarray*}&& \protect{\ft^{\bgl}=\left(\,\Tritab(1&2&3\cr4&5|6\cr7|8&9\cr10)\right),}\\&& \protect{\ft_{\bgl}=\left(\,\Tritab(1&3&5\cr2&4|6\cr7|8&10\cr9)\right)}.\end{eqnarray*}
Given a standard $\bgl$-tableau
$\protect{\fs=\left(\,\Tritab(2&5&7\cr3&8|1\cr4|6&10\cr9)\right)}$, we have
\begin{eqnarray*}&&\mathfrak{S}_{\bgl}=\mathfrak{S}_{\{1,2,3\}}\times \mathfrak{S}_{\{4,5\}}\times\mathfrak{S}_{\{6\}}
  \times\mathfrak{S}_{\{7\}}\times\mathfrak{S}_{\{8,9\}}\times\mathfrak{S}_{\{10\}},\\
&&d(\fs)=(1\,2\,5\,8\,6)(3\,7\,4)(9\,10).\end{eqnarray*}
\end{example}

\begin{definition}\label{Def:x-y-lambda}For $\bgl\in \mpn$, let $a_i=|\gl^{(1)}|+\cdots+|\gl^{(i-1)}|$ for $1\leq i\leq m$ with $a_1=0$.  We define  $m_{\bgl}=u^+_{\bgl}x_{\bgl}$ and $n_{\bgl}=u^-_{\bgl}y_{\bgl}$, where
   \begin{eqnarray*}
 &&  x_{\bgl}=\sum_{w\in\mathfrak{S}_{\bgl}}T_w,\qquad\qquad\qquad u_{\bgl}^+=\prod_{i=2}^{m}\prod_{j=1}^{a_i}(J_j-Q_i);\\
&& y_{\bgl}=\sum_{w\in\mathfrak{S}_{\bgl}}(-q)^{-\ell(w)}T_w,\qquad
  u_{\bgl}^-=\prod_{i=2}^{m}\prod_{j=1}^{a_i}(J_j-Q_{m-i+1}).
  \end{eqnarray*}
 For $\fs,\ft\in\std(\bgl)$, we define
 \begin{eqnarray*}
 m_{\fs\ft}=T^*_{d(\fs)}m_{\bgl}T_{d(\ft)},&\qquad&
 n_{\fs\ft}=\!(\!-\!q\!)^{-\!\ell(d(\fs))\!-\!\ell(d(\ft))}T^*_{d(\fs)}n_{\bgl}T_{d(\ft)}.
 \end{eqnarray*}
 \end{definition}

Lemma~\ref{Lemm:J-property} shows $m_{\bgl}=u^+_{\bgl}x_{\bgl}=x_{\bgl}u_{\bgl}^+$ and $n_{\bgl}=u^-_{\bgl}y_{\bgl}=y_{\bgl}u^-_{\bgl}$. Hence $m_{\bgl}^*=m_{\bgl}$ and $n_{\bgl}^*=n_{\bgl}$. Note that $m_{\bgl}=m_{\ft^{\bgl}\ft^{\bgl}}$, $m_{\fs\ft}^*=m_{\ft\fs}$, and $n_{\fs\ft}^*=n_{\ft\fs}$.

\begin{remark}\label{Remark:mn-def} Clearly  $x_{\bgl}$, $y_{\bgl}$, $u^{\pm}_{\bgl}$ are also well-defined when $\bgl$ is a multicomposition. Thus  $m_{\fs\ft}$ and $n_{\fs\ft}$ are well-defined for row standard $\bgl$-tableaux $\fs,\ft$.\end{remark}

\begin{theorem}[\protect{\cite[Theorem~3.26]{DJM}, \cite[(3.1)]{Mathas}}]\label{Them:mn-basis}
The algebra $H$ is a free $R$-module with cellular basis  $\mathcal{M}=\{m_{\fs\ft}|\fs,\ft\in\mathrm{std}(\bgl)
  \text{ for }\bgl\vdash n\}$ and $
 \mathcal{N}=\{n_{\fs\ft}|\fs,\ft\in\mathrm{std}(\bgl)\text{ for }\bgl\vdash n\}$.
\end{theorem}

\section{$(\bk,\bl)$-semistandard tableaux}\label{Sec:super-tableaux}
In this section, we define the $(\bk,\bl)$-semistandard tableaux associated to ($(\bk,\bl)$-hook) multipartitions  of $n$, which play an important role in later sections. Let us remark that  we are in the classical settings when $\bl=\mathbf{0}$.

\begin{point}{}* Let $\mathbb{N}$ be the set of non-negative integers. Denote by $\mathcal{C}(k+\ell, n)$ and by $\mathcal{P}(k+\ell, n)$ the set of compositions and partitions  of $n$ with $k+\ell$ parts respectively. For $n_1, n_2\in \mathbb{N}$, $\lambda\in \mathcal{C}(k,n_1)$ and $\mu\in \mathcal{C}(\ell,n_2)$, we set
\begin{eqnarray*}
  &&\lambda\!\vee\!\mu=(\lambda;\mu)\in \mathcal{C}(k+\ell, n_1+n_2).
\end{eqnarray*}
Every element in $\mathcal{C}(k+\ell, n)$ has the form $\lambda\vee\mu$ for some $\lambda\in\mathcal{C}(k, n_1)$, $\mu\in\mathcal{C}(\ell, n_2)$ with $n_1+n_2=n$.
Furthermore,  we let \begin{eqnarray*}
         &&\mathcal{C}(k|\ell;n)=\left\{\lambda|\mu: \begin{array}{l}\exists n_1,n_2\in\mathbb{N}, \lambda\in \mathcal{C}(k,n_1),\mu\in \mathcal{C}(\ell,n_2), \lambda\vee \mu\in  \mathcal{C}(k+\ell, n)\end{array}\right\},\\
&&\mathcal{P}(k|\ell;n)=\{\lambda|\mu\in \mathcal{C}(k|\ell,n):  \lambda,\mu\text{ are partitions}\}.
       \end{eqnarray*}
Clearly the map $\lambda|\mu\mapsto \lambda\vee\mu$ is a bijection between $\mathcal{C}(k|\ell,n)$ and $\mathcal{C}(k+\ell,n)$.

Following \cite[Definition~2.3]{B-Regev}, a partition $\gl=(\gl_1, \gl_2, \cdots)\vdash n$ is said to be a \textit{$(k, \ell)$-hook partition} of $n$ if $\gl_{k+1}\leq \ell$. We let  $H(k|\ell;n)$ denote the set of all $(k,\ell)$-hook partitions of $n$
 and let
\begin{equation*}
 \mathcal{P}^+(k|\ell;n)=\left\{\mu|\nu\in  \mathcal{P}(k|\ell;n): \mu_k\geq \ell(\nu)\right\}.
\end{equation*}
\end{point}

The following fact will be useful.

\begin{lemma}[\cite{BL}]Keeping notations as above. Then there is a bijection between $H(k|\ell;n)$ and $\mathcal{P}^{+}(k|\ell;n)$ given by $\gl\mapsto \gl_{\sharp}|\gl_{\ast}$, where
\begin{eqnarray*}
\gl_{\sharp}=(\gl_1,\ldots, \gl_k), &\quad\quad&\gl_{\ast}=(\gl_{k+1},\gl_{k+2},\ldots)^{'}.
\end{eqnarray*}
\label{Lemm:hook=bi-partitions}\end{lemma}

\begin{point}{}*\label{Subsec:linearly-order}Let $\mathbf{x}, \mathbf{y}$ be sets of $k, \ell$ symbols respectively as follows
\begin{eqnarray*}
  \mathbf{x}^{(i)}&=&\left\{x^{(i)}_1, \ldots, x_{k_i}^{(i)}\right\},\quad 1\leq i\leq m;\\
  \mathbf{y}^{(i)}&=&\left\{y^{(i)}_1, \ldots, y_{\ell_i}^{(i)}\right\},\quad 1\leq i\leq m;\\
    \mathbf{x}&=&\mathbf{x}^{(1)}\cup\cdots\cup \mathbf{x}^{(m)};\\
    \mathbf{y}&=&\mathbf{y}^{(1)}\cup\cdots\cup \mathbf{y}^{(m)}.
\end{eqnarray*}
We say that the symbols in $\mathbf{x}^{(i)}\cup \mathbf{y}^{(i)}$ are  \textit{of color $i$} and
define a linearly order on $\bx\cup\by$: \begin{eqnarray*}
&&x_{*}^{(i)}<y_{*}^{(i)}<x_{*}^{(i+1)}\text{ for }1\leq i<m;\\
&&x_{a}^{(i)}<x_{b}^{(j)}\text{ if and only if } i<j\text{ or }i=j\text{ and }a<b;\\
 &&  y_{a}^{(i)}<y_{b}^{(j)}\text{ if and only if }i<j\text{ or }i=j\text{ and }a<b.                                                          \end{eqnarray*}
 For simplicity we write $\bx=\{x_1, \ldots, x_k\}$ and $\by=\{y_1, \ldots, y_k\}$ when $m=1$.
\end{point}

 Recall that a {\it skew diagram} is a diagram obtained by removing a smaller Young diagram from a large one that contains it. Thus a Young diagram can be viewed as a special case of a skew Young diagram. In other words, if $\lambda=(\alpha_1, \alpha_2,\ldots), \beta=(\beta_1, \beta_2,\ldots)$ are partitions with $\ell(\lambda)\leq \ell(\beta)$ and $\alpha_i\leq \beta_i$ for $i=1, \ldots, \ell(\alpha)$, then the set theoretic difference $\beta\backslash\alpha$ of the corresponding diagrams forms a so-called {\it skew diagram}.

 A \textit{$(k,\ell)$-skew tableau} $\mathrm{T}$ of shape  $\beta\backslash\alpha$ is a tableau obtained from the skew Young diagram $\beta\backslash\alpha$ by filling the boxes with elements of $\mathbf{x}\cup\mathbf{y}$, allowing repetitions. By a  \textit{$(k,\ell)$-tableau} $\mathrm{T}$ of shape $\lambda$, we means a tableau obtained from the Young diagram $\lambda$ by filling the boxes with elements $\mathbf{x}\cup\mathbf{y}$, allowing repetitions.

 \begin{definition} We say a  $(k,\ell)$-tableau $\mathrm{T}$ is \textit{$(k,\ell)$-semistandard} if
\begin{enumerate}
  \item[(a)] The $\mathbf{x}$-part $\mathrm{T}_\mathbf{x}$ (i.e.,  the boxes filled with symbols $\mathbf{x}$) of $\mathrm{T}$ is a tableau (thus its $\mathbf{y}$-part $\mathrm{T}_{\mathbf{y}}$ is a skew tableau);
  \item[(b)] the entries of $\mathrm{T}_\mathbf{x}$ are nondecreasing in rows, strictly increasing in columns, that is, $\mathrm{T}_\mathbf{x}$ is a \textit{row-semistandard tableau};
  \item[(c)] the entries of $\mathrm{T}_\mathbf{y}$ are nondecreasing in columns, strictly increasing in rows, that is, $\mathrm{T}_\mathbf{y}$ is a \textit{column-semistandard skew tableau}.
\end{enumerate}
\end{definition}

It is known that  a $\lambda$-tableau can be a $(k,\ell)$-semistandard tableau if and only if $\lambda$ is $(k,\ell)$-hook partition (see \cite[\S4.2]{B-Regev} or  \cite[Lemma~4.2]{BKK}).

\begin{example}For $\lambda=(3,3,2,2,1)\in H(2|2;11)$, we let
\begin{equation*}\mathrm{T}=\diagram(x_1&x_1&x_1\cr x_2&y_1&y_2\cr y_1&y_2\cr y_1&y_2\cr y_1)   \qquad \text{ and }\qquad \mathrm{S}=\diagram(x_1&x_1&x_1\cr x_2&y_1&x_2\cr y_1&y_2\cr y_1&y_2\cr y_1).\end{equation*}
Then $\mathrm{T}$ is a $(2,2)$-semistandard $\lambda$-tableau, while $\mathrm{S}$ is not a $(2,2)$-semistandard $\lambda$-tableau.
\end{example}

\begin{definition}
Let $\gl$ be a partition of $n$ and $\mu|\nu\in\mathcal{C}(k|\ell;n)$.  A $(k,\ell)$-tableau $\mathrm{T}$ of shape $\lambda$ is said to be \textit{of type $\mu|\nu$} if the number of the symbol $x_i$ (resp. $y_i$) occurring in $\mathrm{T}$ is $\mu_i$ (resp. $\nu_i$). We say that a
$(k,\ell)$-tableau $\mathrm{T}$ of type $\mu|\nu$ is a $(k,\ell)$-\textit{semistandard} tableau of type $\mu|\nu$ if $\mathrm{T}_{\mathbf{x}}$ is a row-semistandard tableau of type $\mu$ (thus $\mathrm{T}_{\mathbf{y}}$ is a column-semistandard skew tableau of type $\nu$).
\end{definition}
We let $\std(\gl,\mu|\nu)$ be the set of all $(k,\ell)$-semistandard $\gl$-tableaux of type $\mu|\nu$. Note that if $\nu=0$ then a $(k,\ell)$-semistandard $\lambda$-tableau $S$ of type $\mu|\nu$ is exactly a semistandard $\gl$-tableau of type $\mu$ filling elements in $\bx$ in the classical setting. Moreover, if $\mathrm{T}\in \std(\gl,\mu|\nu)$ then the subtableau obtained by removing the top $k$ rows from $\mathrm{T}$ is a column-semistandard tableau filling elements in $\by$. Clearly, a $(k,\ell)$-semistandard $\lambda$-tableau is a $(k,\ell)$-semistandard $\lambda$-tableau of type $\mu|\nu$ for some $\mu|\nu\in\mathcal{C}(k|\ell;n)$.

Note that if $\gl\in H(k|\ell;n)$, then there is a unique $(k,\ell)$-semistandard $\gl$-tableau $\mathrm{T}^{\gl}$ of type $\gl_{\sharp}|\gl_{\ast}$. Furthermore, we have the
following fact, which is well-known (see \cite[Theorem~2]{Serg} or  \cite[Lemma~4.2]{DR}).

\begin{lemma}\label{Lemm:Std-type}For a partition $\lambda$ of $n$, $\std(\gl,\mu|\nu)\neq\emptyset$ for some $\mu|\nu\in \mathcal{C}(k|\ell;n)$ if and only if $\gl\in H(k|\ell;n)$.
\end{lemma}

\begin{example}Let $\gl=(3,3,2,2,1)\in H(2|2;11)$. Then $(\gl_{\sharp};\gl_{\ast})=((3,3);(3,2))$ and
\begin{equation*}
  \mathrm{T}^{\gl}=\diagram(x_1&x_1&x_1\cr x_2&x_2&x_2\cr y_1&y_2\cr y_1&y_2\cr y_1)
\end{equation*}
\end{example}

\begin{point}{}*Let $\mathcal{C}(\bk;n)$ and $\mathcal{P}(\bk;n)$ be the sets of all multicompositions and of all multipartitions $\bgl=(\gl^{(1)};\cdots; \gl^{(m)})$ of $n$ such that  $\gl^{(i)}$ having $k_i$ parts $i=1, \ldots, m$, respectively. For $n_1, n_2\in\mathbb{N}$, $\bgl\in \mathcal{C}(\bk;n_1)$ and $\bmu\in \mathcal{C}(\bl;n_2)$, we set
\begin{eqnarray*}
  \bgl\!\vee\!\bmu &=&(\lambda^{(1)}\!\vee\!\mu^{(1)};\ldots;\lambda^{(m)}\!\vee\!\mu^{(m)})\in \mathcal{C}(k\!+\!\ell; n_1\!+\!n_2).
\end{eqnarray*}
Further we let \begin{eqnarray*}
         &&\mathcal{C}(\bk|\bl;n)=\{\bgl|\bmu: \exists n_1,n_2\in\mathbb{N}, \bgl\in \mathcal{C}(\bk;n_1),\bmu\in \mathcal{C}(\bl;n_2), \bgl\!\vee\!\bmu\in  \mathcal{C}(k\!+\!\ell; n)\},\\
 &&\mathcal{P}(\bk|\bl;n)=\{\bgl|\bmu\in \mathcal{C}(\bk|\bl;n):  \bgl,\bmu\text{ are multipartitions}\}.
       \end{eqnarray*}
Note that there is a bijection between $\mathcal{C}(\bk|\bl;n)$ and $\mathcal{C}(k+\ell;n)$ defined by $\bgl|\bmu\mapsto \bgl\vee\bmu$.\end{point}

A multipartition $\bgl=(\gl^{(1)}; \ldots; \gl^{(m)})$ of $n$ is said to be a \textit{$(\bk,\bl)$-hook multipartition} of $n$ if $\gl^{(i)}$ is a $(k_i,\ell_i)$-hook partition for $i=1, \ldots,m$. We denote by $H(\bk|\bl; n)$ the set of all  $(\bk,\bl)$-hook multipartitions of $n$. A \textit{$(\bk,\bl)$-tableau} $\mathrm{T}=(\mathrm{T}^{(1)};\ldots; \mathrm{T}^{(m)})$ of shape $\bgl$ is obtained from the diagram of $\bgl$ by inserting the symbols $\mathbf{x}\cup\mathbf{y}$ into its boxes, allowing repetitions.

\begin{definition}  A $(\bk,\bl)$-tableau $\mathrm{T}=(\mathrm{T}^{(1)};\ldots; \mathrm{T}^{(m)})$ is called \textit{$(\bk,\bl)$-semistandard} if
\begin{enumerate}
  \item[(a)] the $\mathbf{x}$-part $\mathrm{T}_\mathbf{x}=(\mathrm{T}^{(1)}_{\mathbf{x}};\ldots; \mathrm{T}^{(m)}_{\mathbf{x}})$ of $\mathrm{T}$ is a tableau with entries in $\bx$ ( thus its $\mathbf{y}$-part $\mathrm{T}_{\mathbf{y}}=(\mathrm{T}^{(1)}_{\mathbf{y}};\ldots; \mathrm{T}^{(m)}_{\mathbf{y}})$  is a skew tableau with entries in $\by$); and
  \item[(b)] $\mathrm{T}_{\mathbf{x}}^{(i)}$ is a row-semistandard tableau and $\mathrm{T}_{\mathbf{y}}^{(i)}$ is a column-semistandard skew tableaux for $i=1, \ldots, m$; and
  \item[(c)] for each $i$ with $1\leq i\leq m$, the color of the symbols in $\mathrm{T}^{(i)}$ is $i$.
\end{enumerate}
\end{definition}

Thus a $(\bk,\bl)$-tableau $\mathrm{T}=(\mathrm{T}^{(1)};\ldots; \mathrm{T}^{(m)})$ of shape
$\bgl$ is  $(\bk,\bl)$-semistandard means that for each $i$ with $1\leq i\leq m$, $\mathrm{T}^{(i)}$ is a $(k_i,\ell_i)$-semistandard tableau with entries of $\bx^{(i)}\cup\by^{(i)}$. Thanks to \cite[\S4.2]{B-Regev} (see also \cite[Lemma~4.2]{BKK}), a $\bgl$-tableau with entries $\bx\cup\by$ can be made into a $(\bk,\bl)$-semistandard tableau if and only if $\bgl$ is a $(\bk,\bl)$-hook multipartition.

\begin{example}\label{Exam:kl-tableau}Let $\bgl=((2,1,1);(3,2,2,1);(4,3,1)\in H(\bk|\bl;20)$ with $\bk|\bl=(1|1,1|2,1|3)$. Then $\mathbf{x}=\{x^{(1)}_1,x^{(2)}_1,x^{(3)}_1\}$, $\mathbf{y}=\{y^{(1)}_1,y^{(2)}_1,y^{(2)}_2,y_1^{(3)}, y_2^{(3)},y_3^{(3)}\}$, and
\begin{equation*}
  \mathrm{T}^{\bgl}=\biggl(\,\Tricolor(x_1^{(1)}&x_1^{(1)}\cr y_1^{(1)}\cr y_1^{(1)} |x_1^{(2)}&x_1^{(2)}&x_1^{(2)}\cr y_1^{(2)}&y_2^{(2)}
  \cr y_1^{(2)}&y_2^{(2)}\cr y_1^{(2)}|x_1^{(3)}&x_1^{(3)}&x_1^{(3)}&x_1^{(3)}
  \cr y_1^{(3)}&y_2^{(3)}&y_3^{(3)}\cr y_1^{(3)})\,\biggr)
\end{equation*}
is a $(\bk,\bl)$-semistandard $\bgl$-tableau
\end{example}

\begin{definition}\label{Def:KL-mu|nu}For $\bgl\in\mpn$ and $\bmu|\bnu\in \mathcal{C}(\bk|\bl;n)$, we say that a $(\bk,\bl)$-tableau $\mathrm{T}=(\mathrm{T}^{(1)};\ldots;\mathrm{T}^{(m)})$ of shape $\bgl$ is \textit{of type $\bmu|\bnu$} if for $a=1, \ldots, m$, the number of the symbol $x^{(a)}_i$, $i=1,\ldots, k_a$, (resp. $y^{(a)}_j$, $j=1,\ldots,\ell_a$)  occurring in  $\mathrm{T}$ is $\mu^{(a)}_i$ (resp. $\nu^{(a)}_j$).
 A $(\bk,\bl)$-tableau $\mathrm{T}=(\mathrm{T}^{(1)};\ldots;\mathrm{T}^{(m)})$ of type $\bmu|\bnu$ is $(\bk,\bl)$-\textit{semistandard} if
\begin{enumerate}
  \item[(a)] the $\mathbf{x}$-part $\mathrm{T}_\mathbf{x}=(\mathrm{T}^{(1)}_{\mathbf{x}};\ldots; \mathrm{T}^{(m)}_{\mathbf{x}})$  of $\mathrm{T}$ is  a tableau (thus its $\mathbf{y}$-part $\mathrm{T}_{\mathbf{y}}=(\mathrm{T}^{(1)}_{\mathbf{y}};\ldots; \mathrm{T}^{(m)}_{\mathbf{y}})$ is a skew tableau);
  \item[(b)] for $i=1,\ldots,m$,  $\mathrm{T}_{\mathbf{x}}^{(i)}$  and $\mathrm{T}_{\mathbf{y}}^{(i)}$ are row-semistandard tableaux and  column-semistandard skew tableaux respectively;
  \item[(c)] for $i=1,\ldots,m$, the colors of the symbols in $\mathrm{T}^{(i)}$ are not less than $i$.
\end{enumerate}
\end{definition}

Denote by $\std(\bgl,\bmu|\bnu)$ the set of $(\bk,\bl)$-semistandard $\bgl$-tableau of type $\bmu|\bnu$. Clearly $\std(\bgl,\bmu|\emptyset)$ is the set of all semistandard $\bgl$-tableau of type $\bmu$ in the sense of \cite[Defintion~4.4]{DJM}. For $\mathrm{S}\in\std(\bgl,\bmu|\bnu)$, we let $\mathrm{S}_{\bx}$ and $\mathrm{S}_{\by}$ be its $\bx$-part and $\by$-part respectively.  Then $\mathrm{S}_{\bx}$ is a row-semistandard tableau of type $\bmu$ and  $\mathrm{S}_{\by}$ satisfies
\begin{enumerate}
  \item[(1)] $\mathrm{Shape}(\mathrm{S}_{\by})=\bgl\backslash \mathrm{Shape}(\mathrm{S}_{\bx})$ is a skew diagram;
  \item[(2)] For  $a=1,\ldots, m$, the number of $y^{(a)}_j$, $j=1,\ldots,  \ell_a$,  occurring in  $\mathrm{S}_{\by}$ is $\nu^{(a)}_j$;
  \item[(3)] the entries of $\mathrm{S}_\mathbf{y}$ are nondecreasing in columns, strictly increasing in rows;
   \item[(4)] the colors of the symbols in $\mathrm{S}_{\by}^{(i)}$ are not less than $i$, $i=1, \ldots, m$,
\end{enumerate}
that is, $\mathrm{S}_{\by}$
is a  column-semistandard skew tableau of type $\bnu$. Sometimes we write $\mathrm{S}=\mathrm{S}_{\bx}|\mathrm{S}_{\by}$

\begin{example}For $\bgl\in H(\bk|\bl;n)$, there is a unique $(\bk,\bl)$-semistandard $\bgl$-tableau $ \mathrm{T}^{\bgl}$ of type $\bgl_{\sharp}|\bgl_{\ast}$. For example,  let $\bgl=((2,1,1);(3,2,2,1);(4,3,1)\in H(\bk|\bl;20)$ with $\bk|\bl=(1|1,1|2,1|3)$. Then $\bgl_{\sharp}=((2);(3,3);(2,1))$, $\bgl_{\ast}=((2);(3,2);(2,1,1))$, and $\mathrm{T}^{\bgl}$ is exactly the ones in  Example~\ref{Exam:kl-tableau}.
  \end{example}

\begin{remark}If  $\bgl\in H(\bk|\bl;n)$ then $\std(\bgl,\bmu|\bnu)\neq \emptyset$ for some $\bmu|\bnu\in \mathcal{C}(\bk|\bl;n)$. Indeed, if $\bgl\in H(\bk|\bl;n)$ then $\std(\bgl,\bgl_{\sharp}|\bgl_*)=\{\mathrm{T}^{\bgl}\}\neq \emptyset$. It is natural to expect the converse also holds, that is, if $\std(\bgl,\bgl_{\sharp}|\bgl_*)\neq \emptyset$ for some $\bmu|\bnu\in \mathcal{C}(\bk|\bl;n)$, then  $\bgl\in H(\bk|\bl;n)$.
\end{remark}

For $\bmu|\bnu\in \mathcal{C}(\bk|\bl;n)$, we let $\mathrm{T}^{\bmu|\bnu}$ be the unique $(k,\ell)$-tableau of shape $\bmu|\bnu$ of type $\bmu|\bnu$ with the symbols $x_j^{(i)}$ (resp. $y_b^{(i)}$) entered in the $a$-th (resp. $b$-th) row of $\mu^{(i)}$ (resp. $\nu^{(i)}$) for $a=1, \ldots, \ell_i$ (resp. $b=1, \ldots, k_i$) and $i=1, \ldots, m$.
\begin{definition}\label{Def:mu|nu(s)}
 Suppose that $\bgl\in H(\bk|\bl;n)$ and $\bmu|\bnu\in \mathcal{C}(\bk|\bl;n)$. For $\fs\in\mathrm{Std}(\bgl)$, we define $\bmu|\bnu(\fs)$ to be the $(\bk,\bl)$-tableau of shape $\bgl$ obtained from $\fs$ by replacing each entries $i\in\fs$ by the symbol of $\mathrm{T}^{\bmu|\bnu}$ filled in the same box of $\ft^{\bmu|\bnu}$ filled $i$.
\end{definition}

\begin{example}\label{Exam:stand-sstd}Let $\bk|\bl=(1|1,1|2)$, $\bgl=((3,1,1);(2,2,1))\in H(\bk|\bl;10)$ and let $\bmu|\bnu=((2);(3))|((2);(2,1))$. Then \begin{equation*}
 \mathfrak{t}^{\bmu|\bnu}=\biggl(\,\Twotab(1&2\cr 3&4|5&6&7\cr8&9\cr10)\,\biggr)\text{ and }
 \mathrm{T}^{\bmu|\bnu}=\biggl(\,\Twocolor(x_1^{(1)}&x_1^{(1)}\cr y_1^{(1)}&y_1^{(1)}|x_1^{(2)}&x_1^{(2)}&x_1^{(2)}\cr y_1^{(2)}&y_1^{(2)}\cr y_2^{(2)})\,\biggr).
\end{equation*}
  Now suppose that
  \begin{equation*}
 \mathfrak{s}=\biggl(\,\Twotab(1&2&5\cr 3\cr4|6&7\cr8&10\cr9)\,\biggr), \quad \tilde{\mathfrak{s}}=\biggl(\,\Twotab(1&2&6\cr 3\cr4|5&7\cr8&10\cr9)\,\biggr),\end{equation*}
 \begin{equation*}
 \mathfrak{s}_1=\biggl(\,\Twotab(1&3&9\cr 2\cr4|5&6\cr7&8\cr10)\,\biggr), \quad \mathfrak{s}_2=\biggl(\,\Twotab(1&4&9\cr 2\cr3|5&7\cr6&8\cr10)\,\biggr).\end{equation*}
  Then
  \begin{equation*}\bmu|\bnu(\mathfrak{s})=\bmu|\bnu(\tilde{\mathfrak{s}})=
 \biggl(\,\Twocolor(x_1^{(1)}&x_1^{(1)}&x_1^{(2)}\cr y_1^{(1)}\cr y_1^{(1)}|x_1^{(2)}&x_1^{(2)}\cr y_1^{(2)}&y_2^{(2)}\cr y_1^{(2)})\,\biggr),\quad\quad\end{equation*}
 \begin{equation*}\bmu|\bnu(\mathfrak{s}_1)=\bmu|\bnu(\mathfrak{s}_2)=
 \biggl(\,\Twocolor(x_1^{(1)}&y_1^{(1)}&y_1^{(2)}\cr x_1^{(1)}\cr y_1^{(1)}|x_1^{(2)}&x_1^{(2)}\cr x_1^{(2)}&y_1^{(2)}\cr y_2^{(2)})\,\biggr).\end{equation*}
  It is easy to see that $\bmu|\bnu(\mathfrak{s})$, $\bmu|\bnu(\tilde{\mathfrak{s}})$ are  $(\bk,\bl)$-semistandard $\bgl$-tableaux of type $\bmu|\bnu$ and  $\bmu|\bnu(\mathfrak{s}_1)$, $\bmu|\bnu(\mathfrak{s}_2)$ are not $(\bk,\bl)$-semistandard $\bgl$-tableaux.
\end{example}

\begin{definition} For $\fs\in\std(\bgl)$ and $1\leq i\leq n$, if $i=(a,b,c)\in \fs$ then let     \begin{equation*}    \mathrm{row}_\fs(i)=\left\{\begin{array}{cc}  x_{a}^{(c)} & \text{if } a\leq k_c; \\ y_{b}^{(c)}& \text{otherwise}. \end{array}    \right.\end{equation*}\end{definition}

\begin{remark}\label{Remark:s=s_x|s_y}
If $\bgl\in H(\bk|\bl; m)$ and $\mathrm{S}\in \std(\bgl,\bmu|\bnu)$, then there exists a standard $\bgl$-tableau $\fs$ with $\mathrm{S}=\bmu|\bnu(\fs)$.  Furthermore, the function $\mathrm{row}_{\fs}$ enables us to recover the standard $\bgl$-tableau $\fs$. In this case, we write $\fs=\fs_{\bmu}|\fs_{\bnu}$
where $\fs_{\bmu}$ and $\fs_{\bnu}$ are the subtableau and skew subtableau of $\fs$ such that $\bmu(\fs_{\bmu})=\mathrm{S}_{\bx}$ and $\bnu(\fs_{\bnu})=\mathrm{S}_{\by}$ respectively.\end{remark}

For $\fs\in\std(\bgl)$ and $i=1, \ldots, n$, it is clear that the nodes of $\fs$ filling entries $1, \ldots, i$ is a Young diagram of a multipartition and denote it by $\fs\downarrow\!i$. For $\fs,\ft\in\std(\bgl)$, we write $\fs\unlhd\ft$ if $\fs\downarrow\!i\unlhd\ft\downarrow\!i$ for all $i=1,\ldots, n$ and write $\fs\lhd\ft$ if $\fs\unlhd\ft$ and $\fs\neq\ft$.
\begin{corollary}\label{Cor:fisrt-last}
  Let $\mathrm{S}\in \std(\bgl, \bmu|\bnu)$. Then there exist standard $\bgl$-tableaux $\mathrm{first}(\mathrm{S})$ and $\mathrm{last}(\mathrm{S})$ such that
  \begin{enumerate}
    \item[(i)] $\bmu|\bnu(\mathrm{first}(\mathrm{S}))=\bmu|\bnu(\mathrm{last}(\mathrm{S}))=\mathrm{S}$ and
    \item[(ii)] if $\fs\in\mathrm{std}(\bgl)$ such that $\bmu|\bnu(\fs)=\mathrm{S}$ then $\mathrm{first}(\mathrm{S})\unrhd \fs\unrhd \mathrm{last}(\mathrm{S})$.
  \end{enumerate}
\end{corollary}

\begin{proof}Thanks to Remark~\ref{Remark:s=s_x|s_y}, there exists a standard
$\bgl$-tableau $\fs$ such that $\bmu|\bnu(\fs)=\mathrm{S}$. Now if there exists an integer $i$ such that $\mathrm{row}_{\fs}(i)\neq \mathrm{row}_{\fs}(i+1)$ and $i, i+1$ are in the same row of $\ft^{\bmu|\bnu}$, then the nodes of $\fs$ filled with entries $i, i+1$ are not in the same row (column) of a component of $\fs$. In fact, if $i, i+1$ are in the same row of a component of $\fs$, then $\mathrm{row}_{\fs}(i)\neq \mathrm{row}_{\fs}(i+1)$ means $\mathrm{row}_{\fs}(i)\cup\mathrm{row}_{\fs}(i+1)\in \by$, which implies that $i, i+1$ are not in the same row of $\ft^{\bmu|\bnu}$ owing to $\bmu|\bnu(\fs)\in \std(\bgl,\bmu|\bnu)$;
 While they are in the same column of a component of $\fs$, then $\mathrm{row}_{\fs}(i)\neq \mathrm{row}_{\fs}(i+1)$ means either $\mathrm{row}_{\fs}(i)\cup\mathrm{row}_{\fs}(i+1)\in \bx$ or $\mathrm{row}_{\fs}(i)\in\bx$ and $\mathrm{row}_{\fs}(i+1)\in\by$, which implies that $i, i+1$ are not in the same column of $\ft^{\bmu|\bnu}$ due to $\bmu|\bnu(\fs)\in \std(\bgl,\bmu|\bnu)$. As a consequence, the tableau $\tilde{\fs}=\fs(i,\,i+1)$ is standard and $\bmu|\bnu(\tilde{\fs})=\mathrm{S}$. Also, $\fs\rhd\tilde{\fs}$ if $\mathrm{row}_{\fs}(i) < \mathrm{row}_{\fs}(i+1)$ and $\tilde{\fs}\rhd\fs$ otherwise. Therefore we complete the proof.
\end{proof}

\section{Permutation Supermodules}\label{Sec:Permutation-supermodules}
This section introduces the permutation supermodules of the cyclotomic Hecke algebra and provide a special (homogeneous) basis of these modules, which will used to determine the homomorphisms between permutation supermodules in \S\ref{Sec:Cyclotomic-Schur-superalgebras}.

\begin{definition}\label{Def:Supermodule}For a multicomposition $\bmu$, we write $\widetilde{\bmu}=(1^{|\bmu^{(1)}|}; \ldots;1^{|\bmu^{(m)}|})$.
For $\bmu|\bnu\in C(\bk|\bl;n)$, we let $\mathfrak{S}_{\bmu|\bnu}:=\mathfrak{S}_{\bmu\vee\bnu}\cong \mathfrak{S}_{\bmu}\times \mathfrak{S}_{\bnu}$, $\bmu^*:=\bmu|\widetilde{\bnu}$ and $\bnu_*:=\widetilde{\bmu}|\bnu$. Define
\begin{equation*}
  x_{\bmu*}=\sum_{w\in \mathfrak{S}_{\bmu^*}}T_w \qquad \text{ and }\qquad  y_{\bnu_*}=\sum_{w\in \mathfrak{S}_{\bnu_*}}(-q)^{\ell(w)}T_w.
\end{equation*}
The $H$-module $M_{\bmu|\bnu}=m_{\bmu^*}n_{\bnu_*}H$ is called the \textit{permutation supermodule} of $H$.
\end{definition}

Clearly $\mathfrak{S}_{\bmu|\bnu}=\mathfrak{S}_{\bmu^*}\mathfrak{S}_{\bnu_*}$ and $x_{\bmu^*}y_{\bnu_*}=y_{\bnu_*}x_{\bmu^*}$. Let us remark that the permutation supermodule $M_{\bmu|\bnu}$ is generated by $m_{\bmu|\bnu}=m_{\bmu^*}n_{\bnu_*}$, which is a super analogues of \cite[Definition~3.8]{DJM}, or a cyclotomic analogues of the signed $q$-permutation modules of $H_n(q)$ defined in \cite[\S5]{DR}.

\begin{point}{}*\label{Point:Dec}
 Given $\bgl\in\mpn$ and  $\mathrm{S}\in \std(\bgl, \bmu|\bnu)$, if $\mathrm{S}=\bmu|\bnu(\fs)$ for some $\fs\in\mathrm{std}(\bgl)$ then, thanks to Remark~\ref{Remark:s=s_x|s_y}, we may write $\fs=\fs_{\bmu}|\fs_{\bnu}$ such that $\bmu(\fs_{\bmu})=\mathrm{S}_{\bx}$ and $\bmu(\fs_{\bnu})=\mathrm{S}_{\by}$. For $\ft\in\mathrm{std}(\bgl)$, we denote by $\ft_{\bmu}$ (resp. $\ft_{\bnu}$) the standard subtableau (resp. standard skew subtableau) of $\ft$ with shape $\mathrm{shape}(\fs_{\bmu})$ (resp.  $\mathrm{shape}(\fs_{\bnu})$). Note that the standard (sub)tableau $\ft_{\bmu}$ may and will be viewed as the unique row-standard $\mathrm{shape}(\ft_{\bmu})|\tilde{\bnu}$-tableau $\ft_{\bmu}|\ft^{\widetilde{\bnu}}$, where $\ft^{\widetilde{\bnu}}$ is the $\widetilde{\bnu}$-tableau with the numbers occurring in  $\ft_{\bnu}$ entered in order first along the rows of the first component of $\widetilde{\bnu}$, and then along the rows of the second component, and so on. Similarly, the standard skew subtableau $\ft_{\bnu}$ may and will be viewed as the unique row-standard $\widetilde{\bmu}|\mathrm{shape}(\ft_{\bnu})$-tableau $\ft^{\widetilde{\bmu}}|\ft_{\bnu}$, where $\ft^{\widetilde{\bmu}}$ is the $\widetilde{\bmu}$-tableau with the numbers occurring in $\ft_{\bmu}$ entered in order  first along the rows of the first component of $\widetilde{\bnu}$, and then along the rows of the second component, and so on. Thus we may  and will write $\ft=\ft_{\bmu}|\ft_{\bnu}$ with $\mathrm{shape}(\ft_{\bmu})=\mathrm{shape}(\fs_{\bmu})$ and  $\mathrm{shape}(\ft_{\bnu})=\mathrm{shape}(\fs_{\bnu})$.

 Now we let $\ft^{\mathrm{shape}(\ft_{\bmu})}$ (resp. $\ft^{\mathrm{shape}(\ft_{\bnu})}$) be  the
$\mathrm{shape}(\ft_{\bmu})$-tableau (resp.  skew $\mathrm{shape}(\ft_{\bnu})$-tableau) with the numbers occurring in the subtableau $\ft_{\bmu}$ (resp. skew subtableau $\ft_{\bnu}$) entered in
order first along the rows of the first component, and then along the rows of the second component, and so on. Finally, let $d(\ft_{\bmu})$ (resp. $d(\ft_{\bnu})$) be the element of $\mathfrak{S}_{\mathrm{shape}(\ft_{\bmu})}$ (resp. $\mathfrak{S}_{\mathrm{shape}(\ft_{\bnu})}$) such that $\ft_{\bmu}=d(\ft_{\bmu})\ft^{\mathrm{shape}(\ft_{\bmu})}$ (resp. $\ft_{\bnu}=d(\ft_{\bnu})\ft^{\mathrm{shape}(\ft_{\bnu})}$). Clearly,  $d(\ft_{\bmu})$ (resp. $d(\ft_{\bnu})$) is exactly the element of $\mathfrak{S}_{\mathrm{shape}(\ft_{\bmu})|\widetilde{\bnu}}$ (resp. $\mathfrak{S}_{\widetilde{\bmu}|\mathrm{shape}(\ft_{\bnu})}$) such that $\ft_{\bmu}|\ft^{\widetilde{\bnu}}=d(\ft_{\bmu})\ft^{\mathrm{shape}(\ft_{\bmu})}|\ft^{\widetilde{\bnu}}$ (resp. $\ft^{\widetilde{\bmu}}|\ft_{\bnu}=d(\ft_{\bnu})\ft^{\widetilde{\bmu}}|\ft^{\mathrm{shape}(\ft_{\bnu})}$).
\end{point}

The following easy verified fact will be useful.

\begin{lemma}\label{Lemm:shape(bmu)=shape(bnu)}Let $\bgl\in\mpn$ and $\bmu|\bnu\in \mathcal{C}(\bk|\bl;n)$.  If $\fs, \ft$ are standard $\bgl$-tableaux satisfying $\bmu|\bnu(\fs)=\bmu|\bnu(\ft)=\mathrm{S}\in \std(\bgl,\bmu|\bnu)$, then $\mathrm{shape}(\fs_{\bmu})=\mathrm{shape}(\ft_{\bmu})$ and $\mathrm{shape}(\fs_{\bnu})=\mathrm{shape}(\ft_{\bnu})$.
\end{lemma}

\begin{proof}It suffices to show that $\mathrm{shape}(\fs_{\bmu})=\mathrm{shape}(\ft_{\bmu})$. Thanks to Corollary~\ref{Cor:fisrt-last}, there exist $\mathrm{first}(\mathrm{S}), \mathrm{last}(\mathrm{S})\in \std(\bgl)$ such that $\bmu|\bnu(\mathrm{first}(\mathrm{S}))=\bmu|\bnu(\mathrm{last}(\mathrm{S}))=\mathrm{S}$ and $\mathrm{first}(\mathrm{S})\unrhd\fs,\ft\unrhd\mathrm{last}(\mathrm{S})$. Thus there exist $w_{\fs},w_{\ft}\in \mathfrak{S}_n$ such that $\fs=\mathrm{first}(\mathrm{S})w_{\fs}$ and $\ft=\mathrm{first}(\mathrm{S})w_{\ft}$, this shows $\ft=\fs w$ for some $w\in \mathfrak{S}_n$. Now let $s_{i_1}\cdots s_{i_j}$ be a reduced expression of $w$. Then $\ft_a=\fs s_{i_1}\cdots s_{i_a}$, $a=1, \ldots, j$, are standard $\bgl$-tableau such that $\bmu|\bnu(\ft_i)=\mathrm{S}$. Therefore we may assume that $\ft=\fs(i,i+1)$ for some $i$ such that $i,i+1$ in the same row of $\ft^{\bmu|\bnu}$. Note that the symbols filled in the same row of $\mathrm{T}^{\bmu|\bnu}$ are same. Therefore $\mathrm{shape}(\fs_{\bmu})=\mathrm{shape}(\ft_{\bmu})$.
\end{proof}
Thanks to Corollary~\ref{Cor:fisrt-last} and Remark~\ref{Remark:mn-def}, we can get the following definition, which is a super analogue of \cite[Definition~4.8]{DJM}.
 \begin{definition}\label{Def:St}
   For $\mathrm{S}\in \std(\bgl,\bmu|\bnu)$ and $\ft\in\mathrm{std}(\bgl)$, we define
 \begin{equation*}
 m_{\mathrm{S}\ft}:=\sum_{\fs\in\mathrm{std(\bgl)}, \bmu|\bnu(\fs)=\mathrm{S}}m_{\fs_{\bmu}\ft_{\bmu}}\sum_{\fs\in\mathrm{std(\bgl)}, \bmu|\bnu(\fs)=\mathrm{S}}n_{\fs_{\bnu}\ft_{\bnu}}\end{equation*}
 \end{definition}

\begin{example}\label{Exam:tableau-type}Assumptions and  notations being as in  Example~\ref{Exam:stand-sstd}, let \begin{equation*}
 \mathrm{S}=\biggl(\,\Twocolor(x_1^{(1)}&x_1^{(1)}&x_1^{(2)}\cr y_1^{(1)}\cr y_1^{(1)} |x_1^{(2)}&x_1^{(2)}
  \cr y_1^{(2)}&y_2^{(2)}\cr y_1^{(2)})\,\biggr)\in \std(\bgl,\bmu|\bnu).\end{equation*}
  Then there are three standard $\bgl$-tableaux $\fs$ such that $\bmu|\bnu(\fs)=\mathrm{S}$:
  \begin{equation*}
 \fa=\biggl(\,\Twotab(1&2&5\cr 3\cr 4|6&7
  \cr 8&10\cr 9)\,\biggr),
  \fb=\biggl(\,\Twotab(1&2&6\cr 3\cr 4|5&7
  \cr 8&10\cr 9)\,\biggr), \fc=\biggl(\,\Twotab(1&2&7\cr 3\cr 4|5&6
  \cr 8&10\cr 9)\,\biggr).\end{equation*}
Therefore, according to \S\ref{Point:Dec}, we yield  \begin{eqnarray*}
 &&\fa_{\bmu}\!=\!\protect{\biggl(\!\Twotab(1&2&5|6&7
  )\!\biggr)\!=\!\biggl(\!\Twotab(1&2&5\cr 3\cr 4|6&7
  \cr 8\cr9\cr10)\!\biggr)},
 \fa_{\bnu}\!=\!\protect{\biggl(\!\Twotab(3\cr 4|8&9\cr 10
  )\!\biggr)\!=\!\biggl(\,\Twotab(1\cr2\cr5\cr 3\cr 4|6\cr7
  \cr 8&10\cr9)\,\biggr)};\\
&& \fb_{\bmu}\!\!=\!\!\protect{\biggl(\Twotab(1&2&6|5&7
  )\biggr)\!=\!\biggl(\,\Twotab(1&2&6\cr 3\cr 4|5&7
  \cr 8\cr9\cr10)\biggr)}, \fb_{\bnu}\!=\!\protect{\biggl(\Twotab(3\cr 4|8&9\cr 10
  )\biggr)\!=\!\biggl(\,\Twotab(1\cr2\cr6\cr 3\cr 4|5\cr7
  \cr 8&10\cr9)\,\biggr)};\\
 && \fc_{\bmu}\!\!=\!\!\protect{\biggl(\Twotab(1&2&7|5&6
  )\biggr)\!=\!\biggl(\,\Twotab(1&2&7\cr 3\cr 4|5&6
  \cr 8\cr 9\cr10)\,\biggr)}, \fc_{\bnu}\!=\!\protect{\biggl(\Twotab(3\cr 4|8&9\cr 10
  )\biggr)\!=\!\biggl(\,\Twotab(1\cr2\cr7\cr 3\cr 4|5\cr6
  \cr 8&10\cr9)\,\biggr)}.\end{eqnarray*}
  So, for any $\ft\in\std(\bgl)$, we have
   \begin{equation*}
 m_{\mathrm{S}\ft}=(m_{\fa_{\bmu}\ft_{\bmu}}+m_{\fb_{\bmu}\ft_{\bmu}}+m_{\fc_{\bmu}\ft_{\bmu}}) (n_{\fa_{\bnu}\ft_{\bnu}}
 +n_{\fb_{\bnu}\ft_{\bnu}}+
 n_{\fc_{\bnu}\ft_{\bnu}}).\end{equation*}
 \end{example}
 \begin{lemma}\label{Lemm-basis}
   Suppose that $\bgl$ is a $(\bk,\bl)$-hook multipartition and $\bmu|\bnu\in \mathcal{C}(\bk|\bl;n)$.  If $\mathrm{S}\in \std(\bgl,\bmu|\bnu)$ and $\ft\in \std(\bgl)$ then $m_{\mathrm{S}\ft}\in M_{\bmu|\bnu}$.
 \end{lemma}

\begin{proof}Thanks to Lemma~\ref{Lemm:shape(bmu)=shape(bnu)}, we assume that  $\bga=\mathrm{shape}(\fs_{\bmu})$ and $\bgb=\mathrm{shape}(\fs_{\bnu})$ for all $\fs\in \mathrm{std}(\bgl)$ such that $\bmu|\bnu(\fs)=\mathrm{S}$, where $\bga$ is a multipartition of $|\bmu|$ and $\bgb$ is a multicomposition of $|\bnu|$. Furthermore, according to \S\ref{Point:Dec}, we will identify $\bga$ and $\bgb$ with $\bga^*$ and $\bgb_*$ respectively. Notice that the proof of \cite[Lemma~4.10]{DJM} shows there exists $h\in H$, which commutes with $y_{\bnu_{*}}u_{\bnu_{*}}^-$, such that
 \begin{eqnarray*}
  \sum_{\substack{\fs\in\mathrm{std}(\bgl)\\\bmu|\bnu(\fs)=\mathrm{S}}}
 T_{d(\fs_{\bmu})}^*x_{\bga^{*}}u_{\bga^{*}}^+T_{d(\ft_{\bmu})}&=&
 x_{\bmu^{*}}u_{\bmu^*}^+h\end{eqnarray*}

 Similarly there exists $h'\in H$, which commutes with $x_{\bmu^{*}}u_{\bmu^{*}}^+$, such that
\begin{eqnarray*}
\sum_{\substack{\fs\in\mathrm{std}(\bgl)\\\bmu|\bnu(\fs)=\mathrm{S}}}
\widetilde{T}_{d(\fs_{\bnu})}^* y_{\bgb_{*}}u_{\bgb_{*}}^-\widetilde{T}_{d(\ft_{\bnu})}
  &=&y_{\bnu_*}u_{\bnu_*}^-h'
\end{eqnarray*}
Now by  Definitions~\ref{Def:St} and \ref{Def:x-y-lambda}, we have
\begin{eqnarray*}
  m_{\mathrm{S}\ft}&=&
  \sum_{\substack{\fs\in\mathrm{std}(\bgl)\\\bmu|\bnu(\fs)=\mathrm{S}}}m_{\fs_{\bmu}\ft_{\bmu}}
  \sum_{\substack{\fs\in\mathrm{std}(\bgl)\\\bmu|\bnu(\fs)=\mathrm{S}}}n_{\fs_{\bnu}\ft_{\bnu}}\\
  &=&\sum_{\substack{\fs\in\mathrm{std}(\bgl)\\\bmu|\bnu(\fs)=\mathrm{S}}}
  T^*_{d(\fs_{\bmu})}x_{\bga^*}u_{\bga^*}^+T_{d(\ft_{\bmu})}
  \sum_{\substack{\fs\in\mathrm{std}(\bgl)\\\bmu|\bnu(\fs)=\mathrm{S}}}
  \widetilde{T}^*_{d(\fs_{\bnu})}y_{\bgb_{*}}u_{\bgb_{*}}^-\widetilde{T}_{d(\ft_{\bnu})}\\
   &=&(x_{\bmu^{*}}u_{\bmu^{*}}^+h) (y_{\bnu_{*}}u_{\bnu_{*}}^-h')\\
   &=&(x_{\bmu^{*}}u_{\bmu^{*}}^+)(y_{\bnu_{*}}u_{\bnu_{*}}^-)hh'.
 \end{eqnarray*}
Thus we prove that $m_{\mathrm{S}\ft}\in M_{\bmu|\bnu}$.
\end{proof}

For $h\in H$, we let $h=\sum_{\fs,\ft}a_{\fs\ft}m_{\fs\ft}=\sum_{\fs,\ft}b_{\fs\ft}n_{\fs\ft}$, where  $a_{\fs\ft}, b_{\fs\ft}\in R$, be the unique expression for $h$ in terms of standard basis owing to Theorem~\ref{Them:mn-basis}. We say that $m_{\fs\ft}$ (resp. $n_{\fs\ft}$) is \textit{involved} in $h$ if $a_{\fs\ft}\neq 0$ (resp. $b_{\fs\ft}\neq 0$).
 \begin{corollary}\label{Cor:basis}
   For $\bmu|\bnu\in \mathcal{C}(\bk|\bl;n)$, $u^+_{\bmu^*}H\cap  x_{\bmu^*}H\cap u^-_{\bnu_*}H\cap y_{\bnu_*}H$ is spanned by
   \begin{equation*}
     \left\{m_{\mathrm{S}\ft}\left|\mathrm{S}\in \std(\bgl,\bmu|\bnu), \ft\in\mathrm{std}(\bgl)\text{ for some }\bgl\in H(\bk|\bl;n)\right.\right\}
   \end{equation*}
   In particular, $M_{\bmu|\bnu}=u^+_{\bmu^*}H\cap  x_{\bmu^*}H\cap u^-_{\bnu_*}H\cap y_{\bnu_*}H$.
 \end{corollary}

\begin{proof}  For each multipartition $\bgl$ let
\begin{equation*}
  \std_{\mathrm{last}}(\bgl,\bmu|\bnu)=\{\fs\in\mathrm{std}(\bgl)|\bmu|\bnu(\fs)=\mathrm{last}(\mathrm{S})
  \text{ for some }\mathrm{S}\in \std(\bgl,\bmu|\bnu)\}.
\end{equation*}
Thanks to \cite[Lemma~4.11]{DJM}, every non-zero element  $h\in u_{\bmu^*}^+H\cap x_{\bmu^*}H$ involves a standard basis element $m_{\fs_{\bmu}\ft_{\bmu}}$ with $\fs\in \std_{\mathrm{last}}(\bgl,\bmu|\bnu)$ and $\ft\in\std(\bgl)$ for some multipartition $\bgl$. So there is a non-zero element of $u^-_{\bnu_*}H\cap y_{\bnu_*}H$ involves an element $n_{\fs_{\bnu}\ft_{\bnu}}$.

Now suppose that $h\in u^+_{\bmu^*}H\cap  x_{\bmu^*}H\cap u^-_{\bnu_*}H\cap y_{\bnu_*}H$ and $h=\sum a_{\fa\fb,\fu\fv}m_{\fa\fb}n_{\fu\fv}$ with $a_{\fa\fb,\fu\fv}\in R$. Let
 \begin{equation*}
   \tilde{h}=h-\sum_{\bgl}\sum_{\fs\in\mathcal{T}_{\mathrm{last}}(\bgl,\bmu|\bnu)}
   \sum_{\ft\in\mathrm{std}(\bgl)}r_{\fs\ft}m_{\fs_{\bmu}\ft_{\bmu}}n_{\fs_{\bnu}\ft_{\bnu}}.
 \end{equation*}
 Then $\tilde{h}\in u^+_{\bmu^*}H\cap  x_{\bmu^*}H\cap u^-_{\bnu_*}H\cap y_{\bnu_*}H$, while $\tilde{h}$ does not involved any term $m_{\fs_{\bmu}\ft_{\bmu}}n_{\fs_{\bnu}\ft_{\bnu}}$ for any $\fs\in \mathcal{T}_{\mathrm{last}}(\bgl,\bmu|\bnu)$. Therefore, $\tilde{h}=0$ and the first assertion follows.

Lemma~\ref{Lemm-basis} shows $m_{\mathrm{S}\ft}\in M_{\bmu|\bnu}$. Note that
\begin{equation*}
  M_{\bmu|\bnu}=u^+_{\bmu^*}x_{\bmu^*}u^-_{\bnu_*}y_{\bnu_*}H\subset u^+_{\bmu^*}H\cap  x_{\bmu^*}H\cap u^-_{\bnu_*}H\cap y_{\bnu_*}H.
\end{equation*}
 As a consequence, we yield $M_{\bmu|\bnu}=u^+_{\bmu^*}H\cap  x_{\bmu^*}H\cap u^-_{\bnu_*}H\cap y_{\bnu_*}H$.
\end{proof}

Let us remark that \cite[Corollary~4.13]{DJM} proves $u^+_{\bmu^*}x_{\bmu^*}H=u^+_{\bmu^*}H\cap  x_{\bmu^*}H$. Similarly we have $u^-_{\bnu_*}y_{\bnu_*}H=u^-_{\bnu_*}H\cap y_{\bnu_*}H$. Thus $M_{\bmu|\bnu}=u^+_{\bmu^*}x_{\bmu^*}H\cap u^-_{\bnu_*}y_{\bnu_*}H=M_{\bmu^*}\cap M_{\bnu_*}$ where $M_{\bmu^*}=m_{\bmu^*}H$ and $M_{\bnu_*}=m_{\bnu_*}H$.

Due to \cite[Proposition~3.22]{DJM}, for a multipartition  $\bgl$ of $n$, the $R$-module
\begin{equation*}
  N^{\rhd\bgl}:=\mathrm{Span}_{R}\left\{m_{\fa\fb}|\fa,\fb\in\std(\bga) \text{ for multipartition $\bga$ of $n$ with } \bga\rhd \bgl\right\}
\end{equation*}
is a two-sided ideal of $H$.

Further we let $z_{\bgl}=(N^{\rhd\bgl}+m_{\bgl})/N^{\rhd\bgl}$. The {\it Specht module} $S^{\bgl}$ is the submodule of $H/N^{\rhd\bgl}$ generated by $z_{\bgl}$, which is a free $R$-module with basis $\{z_{\bgl}T_{d(\ft)}|\ft\in\std(\bgl)\}$.

The next fact is a super-version of \cite[Theorem~4.14 and Corollary~4.15]{DJM} or a cyclotomic analogue of \cite[Proposition~5.2]{DR}.
 \begin{theorem}\label{Them:M-basis}Given $\bmu|\bnu\in\mathcal{C}(\bk|\bl;n)$,
   $M_{\bmu|\bnu}$ is a free $R$-module with basis
   \begin{equation*}
     \left\{m_{\mathrm{S}\ft}|\mathrm{S}\in \std(\bgl,\bmu|\bnu),\ft\in\std(\bgl) \text{ for some }\bgl\in\mpn \right\}.
   \end{equation*}
   Furthermore, there exists a filtration of $M_{\bmu|\bnu}$
   \begin{equation*}
     M_{\bmu|\bnu}=M_1>M_2>\cdots>M_{r+1}=0
   \end{equation*}
   such that for each $i$ with $1\leq i\leq r$, $M_{i}/M_{i+1}\cong S^{\bgl_i}$ for some $\bgl_i\vdash n$. Moreover, if $\bgl\in H(\bk|\bl;n)$ then the cardinality of $S^{\bgl}$ in $M_{\bmu|\bnu}$ is $\left|\std(\bgl,\bmu|\bnu)\right|$.
 \end{theorem}
\begin{proof}The first part of the theorem follows by applying Lemma~\ref{Lemm-basis} and Corollary~\ref{Cor:basis}. For the second part, we choose an ordering $\mathrm{S}_1>\mathrm{S}_2>\cdots>\mathrm{S}_r$ on the set of $(\bk,\bl)$-semistandard tableaux of type $\bmu|\bnu$ such that $j>i$ if $\bgl_i\rhd \bgl_j$ where $\mathrm{S}_i\in \std(\bgl_i, \bmu|\bnu)$ and $\mathrm{S}_j\in \std(\bgl_j, \bmu|\bnu)$. Let $M_i$ be the $R$-submodule of $M_{\bmu|\bnu}$ with basis $\{m_{\mathrm{S}_j\ft}|j\geq i,\ft\in\std(\bgl_j)\}$. Then
\begin{equation*}
     M_{\bmu|\bnu}=M_1>M_2>\cdots>M_{r+1}=0
   \end{equation*}
and each $M_i$ is a right ideal due to Theorem~\ref{Them:mn-basis}.

For $i=1, \ldots, r$, we have $M_i\cap N^{\rhd\bgl_i}\subseteq M_{i+1}$. So there is a well-defined $H$-homomorphism
\begin{equation*}
  \theta: S^{\bgl_i}\rightarrow M_i/M_{i+1}, \quad z_{\bgl_i}\mapsto m_{\mathrm{S}_{i}\ft^{\bgl_i}}+M_{i+1}.
\end{equation*}
Since both $S^{\bgl_i}$ and  $M_i/M_{i+1}$ have the rank equal to the number of standard $\bgl_i$-tableaux, $\theta$ is an isomorphism. It completes the proof.
\end{proof}

\section{The Cyclotomic $q$-Schur superalgebra}\label{Sec:Cyclotomic-Schur-superalgebras}
This section devotes to introduce the  cyclotomi $q$-Schur superalgebra and show it is a cellular (super)algebra based on the description of homomorphisms between permutation supermodules.
\begin{definition}\label{Def:Cyclotomi-Schur}
The \textit{cyclotomic $q$-Schur superalgebra} is the endomorphism algebra
\begin{equation*}
  \mathscr{S}(\bk|\bl;n):=\End_{H}\biggr(\bigoplus_{\bmu|\bnu\in \mathcal{C}(\bk|\bl;n)}m_{\bmu|\bnu}H\biggl)
\end{equation*}
with $\mathbb{Z}_2$-grading defined by setting
 \begin{equation*}
   \mathscr{S}(\bk|\bl;n)_i=\bigoplus_{|\bnu|+|\bgb|\equiv i(\text{mod}\,2)}\mathrm{Hom}_{H}( m_{\bmu|\bnu}H, m_{\bga|\bgb}H), \quad i=0,1.
 \end{equation*}
 \end{definition}

\begin{remark}\begin{enumerate}
                \item[(i)]Clearly $\mathscr{S}(\bk|\bl;n)\cong \displaystyle\bigoplus_{\bmu|\bnu,\bga|\bgb\in \mathcal{C}(\bk|\bl;n)}\mathrm{Hom}_{H}(M_{\bmu|\bnu}, M_{\bga|\bgb})$ and $\mathscr{S}(\bk|\bl;n)$ dependents only upon the set $\mathcal{P}(\bk|\bl;n)$ (and not on $\mathcal{C}(\bk|\bl;n)$) up to Morita equivalence.
                \item[(ii)]Given a saturated subset $\Lambda$ of $\mathcal{C}(\bk|\bl;n)$, that is, $\Lambda$ is finite and whenever $\bgl$ is a multipartition such that $\bgl\unrhd \bmu$ for some $\bmu\in\Lambda$ then $\bgl\in\Lambda$,  we may define the cyclotomic $q$-Schur superalgebra associated to $\Lambda$ as following
        \begin{equation*}
  \mathscr{S}(\bk|\bl;\Lambda):=\End_{H}\biggr(\bigoplus_{\bmu|\bnu\in \Lambda}m_{\bmu|\bnu}H\biggl)
\end{equation*}
with $\mathbb{Z}_2$-grading defined by setting
 \begin{equation*}
   \mathscr{S}(\bk|\bl;\Lambda)_i=\bigoplus_{|\bnu|+|\bgb|\equiv i(\text{mod}\,2)}\mathrm{Hom}_{H}( m_{\bmu|\bnu}H, m_{\bga|\bgb}H), \quad i=0,1.
 \end{equation*}
\item[(iii)] Clearly $\mathscr{S}(\bk|\mathbf{0};n)$ is the Dipper-James-Mathas' cyclotomic $q$-Schur algebra in \cite{DJM}.

\item[(iv)]  $\mathscr{S}(\bk|\bl;n)$ is the Du-Rui's quantum Schur superalgebra in \cite{DR} when $m=1$.
              \end{enumerate}\end{remark}

Now we discuss the homomorphisms between permutation supermodules, which enable us to give a cellular basis of the cyclotomic $q$-Schur superalgebra.
\begin{definition}
   Suppose $S$ is a subset of $H$ and define
      \begin{equation*}
     \mathrm{R}(S)=\{h\in H|Sh=0\}\text{ and }\mathrm{L}(S)=\{h\in H|hS=0\}.\end{equation*}
   The \textit{double annihilator} of $S$ is $\mathrm{LR}(S)=\mathrm{L}(\mathrm{R}(S))$.
 \end{definition}

 It is easy to see that if $\bmu$  be a multicomposition  of $n$ then
 \begin{equation*}
  Hu_{\bmu}^+x_{\bmu}\subseteq \mathrm{LR}(u_{\bmu}^+x_{\bmu}) \text{  and  } \mathrm{LR}(u_{\bmu}^+x_{\bmu})\subseteq \mathrm{LR}(u_{\bmu}^+)\cap \mathrm{LR}(x_{\bmu}).
 \end{equation*}
Furthermore, we have the following fact.

 \begin{lemma}\label{Lemm:ann-y}
   Let $\bmu$  be a multicomposition  of $n$. Then
   \begin{eqnarray*}
   \mathrm{LR}(u_{\bmu}^+x_{\bmu})=Hu_{\bmu}^+x_{\bmu}&\text{ and } & \mathrm{LR}(u_{\bmu}^-y_{\bmu})=Hu_{\bmu}^-y_{\bmu}.
   \end{eqnarray*}
   \end{lemma}
\begin{proof}
  The first assertion is exactly \cite[Theorem~5.16]{DJM}. The second assertion follows by applying  a similar argument as the one of \cite[Theorem~5.16]{DJM}.
\end{proof}

Now we can determine the double annihilators of permutation supermodules.

 \begin{theorem}\label{Them:ann}
   Assume that $\bmu|\bnu\in \mathcal{C}(\bk|\bl;n)$. Then $\mathrm{LR}(m_{\bmu|\bnu})=Hm_{\bmu|\bnu}$.
 \end{theorem}

\begin{proof}By applying Lemma~\ref{Lemm:ann-y}, we get
\begin{eqnarray*}
  Hm_{\bmu|\bnu}&=&Hu_{\bmu^*}^+x_{\bmu^*}u_{\bnu_*}^-y_{\bnu_*}\subseteq \mathrm{LR}(m_{\bmu|\bnu})\\
  &\subseteq& \mathrm{LR}(m_{\bmu^*})\cap \mathrm{LR}(n_{\bnu_*})\\
  &=&Hm_{\bmu^*}\cap Hn_{\bnu_*}.
\end{eqnarray*}
While Corollary~\ref{Cor:basis} implies $(m_{\bmu^*}H\cap n_{\bnu_*}H)^*=(m_{\bmu|\bnu}H)^*$. Hence the equality holds and  $\mathrm{LR}(m_{\bmu|\bnu})=Hm_{\bmu|\bnu}$.
\end{proof}

Now the homomorphisms between permutation supermodules can be completely determined by the following
 \begin{corollary}\label{Cor:Hom}
   Assume that $\bmu|\bnu, \bga|\bgb\in \mathcal{C}(\bk|\bl;n)$. Then
   \begin{enumerate}
     \item[(i)] For every $\varphi\in \mathrm{Hom}_{H}(M_{\bmu|\bnu}, M_{\bga|\bgb})$ there exists  $h_{\varphi}\in H$ such that $\varphi(m_{\bmu|\bnu})=h_{\varphi}m_{\bmu|\bnu}$; in particular, $\varphi(m_{\bmu|\bnu})\in M_{\bmu|\bnu}^{*}\cap M_{\bga|\bgb}$.

     \item[(ii)] $\mathrm{Hom}_{H}(M_{\bmu|\bnu}, M_{\bga|\bgb})\cong M_{\bmu|\bnu}^*\cap M_{\bga|\bgb}$.
   \end{enumerate}
 \end{corollary}
\begin{proof}
  Theorem~\ref{Them:ann} and \cite[Lemma~5.2(i)]{DJM} imply that $\varphi(m_{\bmu|\bnu})=h_{\varphi}m_{\bmu|\bnu}$ for some $h_{\varphi}\in H$. Thus $\varphi(m_{\bmu|\bnu})\in M_{\bmu|\bnu}^*$. Clearly $\varphi(m_{\bmu|\bnu})\in M_{\bga|\bgb}$ and we prove (i).  (ii) follows from \cite[Lemma~5.2(ii)]{DJM} via an explicit isomorphism given by $\varphi\mapsto \varphi(m_{\bmu|\bnu})$.
\end{proof}

\begin{point}{}* \label{Point:Double-dec}Now we are ready to construct a basis of $M_{\bmu|\bnu}^*\cap M_{\bga|\bgb}$, or equivalently a basis of $\mathrm{Hom}_{H}(M_{\bmu|\bnu}, M_{\bga|\bgb})$ according to Corollary~\ref{Cor:Hom}(ii). To do that we need more notations. Suppose that $\bmu|\bnu,\bga|\bgb\in \mathcal{C}(\bk|\bl;n)$, $\bgl\in H(\bk|\bl;n)$, $\mathrm{S}\in \std(\bgl,\bmu|\bnu)$ and $\mathrm{T}\in \std(\bgl,\bga|\bgb)$. If $\fs,\ft$ are standard $\bgl$-tableaux such that $\bmu|\bnu(\fs)=\mathrm{S}$ and $\bga|\bgb(\ft)=\mathrm{T}$, then we denote by $\fs_{\bmu\cap\bga}$ (resp. $\ft_{\bmu\cap\bga}$) the standard subtableau of $\fs$ (resp. $\ft$) with shape $\mathrm{shape}(\fs_{\bmu})\cap\mathrm{shape}(\ft_{\bga})$
 and by $\fs_{\bnu\cup\bgb}$ (resp. $\ft_{\bnu\cup\bgb}$) the standard skew-subtableau of $\fs$ (resp. $\ft$) of shape $\bgl/\mathrm{shape}(\fs_{\bmu\cap\bga})$. Let us remark that, thanks to Lemma~\ref{Lemm:shape(bmu)=shape(bnu)}, $\mathrm{shape}(\fs_{\bmu})\cap\mathrm{shape}(\ft_{\bga})$ and $\bgl/\mathrm{shape}(\fs_{\bmu\cap\bga})$ dependents only on $\mathrm{S}$ and $\mathrm{T}$.

Suppose that $\mathrm{shape}(\fs_{\bmu\cap\bga})=\bkp$ and $\bgl/\mathrm{shape}(\fs_{\bmu\cap\bga})=\btau$. Following \S~\ref{Point:Dec},
   we may and will identify the standard (sub)tableau $\fs_{\bmu\cap\bga}$ (resp. $\ft_{\bmu\cap\bga}$) with the unique row-standard $\bkp|\widetilde{\btau}$-tableau $\fs_{\bmu\cap\bga}|\ft^{\widetilde{\btau}}$ (resp. $\ft_{\bmu\cap\bga}|\ft^{\tilde{\btau}}$), where $\ft^{\widetilde{\btau}}$ is the $\widetilde{\btau}$-tableau with the numbers occurring in  $\fs_{\bnu\cup\bgb}$ (resp. $\ft_{\bnu\cup\bgb}$) entered in order first along the rows of the first component of $\widetilde{\btau}$, and then along the rows of the second component, and so on. Similarly, the standard skew subtableau $\fs_{\bnu\cup\bgb}$ (resp. $\ft_{\bnu\cup\bgb}$) may and will be viewed as the unique row-standard $\widetilde{\bkp}|\btau$-tableau $\ft^{\widetilde{\bkp}}|\fs_{\bnu\cup\bgb}$ (resp. $\ft^{\widetilde{\bkp}}|\ft_{\bnu\cup\bgb}$), where $\ft^{\widetilde{\bkp}}$ is the unique row standard $\widetilde{\bkp}$-tableau with the numbers occurring in $\fs_{\bmu\cap\bga}$ (resp. $\ft_{\bmu\cap\bga}$) entered in order first along the rows of the first component of $\widetilde{\bkp}$, and then along the rows of the second component, and so on.
  \end{point}

\begin{example}\label{Exam:Double}Assumptions and notations being as in Example~\ref{Exam:tableau-type}, assume that
$\bga|\bgb=((2);(2))|((2);(2,2))$ and
\begin{equation*}
 \mathrm{T}=\biggl(\,\Twocolor(x_1^{(1)}&x_1^{(1)}&x_1^{(2)}\cr y_1^{(1)}\cr y_1^{(1)} |x_1^{(2)}&y_2^{(2)}
  \cr y_1^{(2)}&y_2^{(2)}\cr y_1^{(2)})\,\biggr)\in \std(\bgl,\bga|\bgb).\end{equation*}
  Then there are two standard $\bgl$-tableaux $\ft$ such that $\bga|\bgb(\ft)=\mathrm{T}$:
  \begin{equation*}
 \tilde{\fa}=\biggl(\,\Twotab(1&2&5\cr 3\cr 4|6&9
  \cr 7&10\cr 8)\,\biggr),
  \tilde{\fb}=\biggl(\,\Twotab(1&2&6\cr 3\cr 4|5&9
  \cr 7&10\cr 8)\,\biggr).\end{equation*}
  Hence, according to \S\ref{Point:Dec}, we obtain
  \begin{eqnarray*}
 &&\fa_{\bmu\cap\bga}=\tilde{\fa}_{\bmu\cap\bga}=\protect{\biggl(\!\Twotab(1&2&5|6
  )\!\biggr)\!=\!\biggl(\!\Twotab(1&2&5\cr 3\cr 4|6\cr7
  \cr 8\cr9\cr10)\!\biggr)}, \\
  &&\fa_{\bnu\cup\bgb}=\tilde{\fa}_{\bnu\cup\bgb}\!=\!\protect{\biggl(\Twotab(3\cr 4|9\cr 7& 10\cr 8
  )\biggr)=\biggl(\,\Twotab(1\cr2\cr5\cr 3\cr 4|6\cr9
  \cr 7&10\cr8)\,\biggr)}.\end{eqnarray*}
     \end{example}

\begin{definition}\label{Def:M-ST}
Assumptions and  notations being as in \S\ref{Point:Double-dec}, we define
  \begin{equation*}
    m_{\mathrm{ST}}=\sum_{\substack{\fs,\ft\in\std(\bgl)\\\bmu|\bnu(\fs)=\mathrm{S}, \bga|\bgb(\ft)=\mathrm{T}}}m_{\fs_{\bmu\cap\bga}\ft_{\bmu\cap\bga}}\sum_{\substack{\fs,\ft\in\std(\bgl)\\\bmu|\bnu(\fs)=\mathrm{S}, \bga|\bgb(\ft)=\mathrm{T}}}n_{\fs_{\bnu\cup\bgb}\ft_{\bnu\cup\bgb}}.
  \end{equation*}
 \end{definition}

 \begin{example}Assumptions and notations being as in Examples \ref{Exam:stand-sstd} and \ref{Exam:Double}. Then
  \begin{equation*}
    m_{\mathrm{ST}}=\sum_{\substack{\fs\in\{\fa,\fb,\fc\}\\ \ft\in\{\tilde{\fa},\tilde{\fb}\}}}m_{\fs_{\bmu\cap\bga}\ft_{\bmu\cap\bga}}
    \sum_{\substack{\fs\in\{\fa,\fb,\fc\}\\ \ft\in\{\tilde{\fa},\tilde{\fb}\}}}n_{\fs_{\bnu\cup\bgb}\ft_{\bnu\cup\bgb}}.
  \end{equation*}
 \end{example}
\begin{remark}\label{Remark:m-ST}For any $\fs,\ft\in\std(\bgl)$ such that $\bmu|\bnu(\fs)=\mathrm{S}$ and $\bga|\bgb(\ft)=\mathrm{T}$, $m_{\fs_{\bmu\cap\bga}\ft_{\bmu\cap\bga}}$ commutes with $n_{\fs_{\bnu\cup\bgb}\ft_{\bnu\cup\bgb}}$. Therefor, we yield that
\begin{eqnarray*}
m_{\mathrm{ST}}^*&=&\sum_{\substack{\fs,\ft\in\std(\bgl)\\\bmu|\bnu(\fs)=\mathrm{S}, \bga|\bgb(\ft)=\mathrm{T}}}m_{\fs_{\bmu\cap\bga}\ft_{\bmu\cap\bga}}^*\sum_{\substack{\fs,\ft\in\std(\bgl)\\\bmu|\bnu(\fs)=\mathrm{S}, \bga|\bgb(\ft)=\mathrm{T}}}n_{\fs_{\bnu\cup\bgb}\ft_{\bnu\cup\bgb}}^*\\
&=&\sum_{\substack{\fs,\ft\in\std(\bgl)\\\bmu|\bnu(\fs)=\mathrm{S}, \bga|\bgb(\ft)=\mathrm{T}}}m_{\ft_{\bmu\cap\bga}\fs_{\bmu\cap\bga}}\sum_{\substack{\fs,\ft\in\std(\bgl)\\\bmu|\bnu(\fs)=\mathrm{S}, \bga|\bgb(\ft)=\mathrm{T}}}n_{\ft_{\bnu\cup\bgb}\fs_{\bnu\cup\bgb}}\\
&=&m_{\mathrm{TS}}.
\end{eqnarray*}
\end{remark}

\begin{proposition}\label{Prop:Basis}Assume that $\bmu|\bnu, \bga|\bgb\in
\mathcal{C}(\bk|\bl;n)$. Then
\begin{equation*}
     \left\{m_{\mathrm{ST}}|\mathrm{S}\in \std(\bgl,\bmu|\bnu), \mathrm{T}\in \std(\bgl,\bga|\bgb) \text{ for some }\bgl \in H(\bk|\bl;n) \right\}
   \end{equation*}
   is a basis of $M_{\bga|\bgb}^*\cap M_{\bmu|\bnu}$.
\end{proposition}
\begin{proof}By Definition~\ref{Def:M-ST} and  Lemma~\ref{Lemm-basis}, we obtain that $m_{\mathrm{ST}}\in M_{\bga|\bgb}^*\cap M_{\bmu|\bnu}$. Moreover $m_{\mathrm{ST}}$ are linearly independently since the elements $m_{\fs_{\bmu\cap\bga}\ft_{\bmu\cap\bga}}n_{\fs_{\bnu\cup\bgb}\ft_{\bnu\cup\bgb}}$ are linearly independent for $\fs,\ft\in \std(\bgl)$ such that $\bmu|\bnu(\fs)=\mathrm{S}$ and $\bga|\bgb(\ft)=\mathrm{T}$. Now suppose that $x\in M_{\bga|\bgb}^*\cap M_{\bmu|\bnu}$. Since $M_{\bga|\bgb}^*=(m_{\bga^*}n_{\bgb_*}H)^*=Hm_{\bga^*}n_{\bgb_*}$ and $M_{\bmu|\bnu}=m_{\bmu^*}n_{\bnu_*}H$, thanks to Theorem~\ref{Them:mn-basis}, we write
 $Hm_{\bga^*}=\sum_{\fa,\fb}r_{\fa,\fb}m_{\fa\fb}$ and $n_{\bgb_*}=\sum_{\tilde{\fa},\tilde{\fb}}\tilde{r}_{\tilde{\fa}\tilde{\fb}}
 n_{\tilde{\fa}\tilde{\fb}}$ (Similarly, we write $m_{\bmu_*}=\sum_{\fa,\fb}r_{\fa,\fb}m_{\fa\fb}$ and $n_{\bnu_*}H=\sum_{\tilde{\fa},\tilde{\fb}}\tilde{r}_{\tilde{\fa}\tilde{\fb}}
 n_{\tilde{\fa}\tilde{\fb}}$). Thus we may assume that
  \begin{equation*}
   x=a_{\fa\fb}^{\tilde{\fa}\tilde{\fb}}\sum_{\fa,\fb,\tilde{\fa},\tilde{\fb}}
   m_{\fa\fb}n_{\tilde{\fa}\tilde{\fb}}
  \end{equation*}
  where $a_{\fa\fb}^{\tilde{\fa}\tilde{\fb}}\in R$. Since $x\in M_{\bmu|\bnu}$, we have $\fa=\fs_{\bmu}$, $\fb=\ft_{\bmu}$,  $\tilde{\fa}=\fs_{\bnu}$ and $\tilde{\fb}=\ft_{\bnu}$ for some $\fs,\ft\in\std(\bgl)$ such that $\bmu|\bnu(\fs)=\mathrm{S}$. Moreover we have $a_{\fa\fb}^{\tilde{\fa}\tilde{\fb}}=a_{\fa'\fb}^{\tilde{\fa'}\tilde{\fb}}$ whenever $\fa=\fs_{\bmu}$ and $\fa'=\fs'_{\bmu}$ for some $\fs,\fs'\in\std(\bgl)$ such that $\bmu|\bnu(\fs)=\bmu|\bnu(\fs')=\mathrm{S}$ due to Theorem~\ref{Them:M-basis}. Conversely
  $x\in M_{\bga|\bgb}^*$ implies $a_{\fa\fb}^{\tilde{\fa}\tilde{\fb}}=a_{\fa\fb'}^{\tilde{\fa}\tilde{\fb'}}$ whenever $\fb=\ft_{\bga}$ and $\fb'=\ft'_{\bga}$ for some $\ft,\ft'\in\std(\bgl)$ such that $\bga|\bgb(\ft)=\bga|\bgb(\ft')=\mathrm{T}$. Thus if $\fa=\fs_{\bmu}, \fa'=\fs'_{\bmu}$ for some $\fs,\fs'\in \std(\bgl)$ with $\bmu|\bnu(\fs)=\bmu|\bnu(\fs')$ and $\fb=\ft_{\bga}, \fb'=\ft'_{\bga}$ for some $\ft,\ft'\in\std(\bgl)$ with $\bga|\bgb(\ft)=\bgb|\bga(\ft')$, then  $a_{\fa\fb}^{\tilde{\fa}\tilde{\fb}}=a_{\fa'\fb}^{\tilde{\fa'}\tilde{\fb}}=
  a_{\fa\fb}^{\tilde{\fa}\tilde{\fb}}=a_{\fa\fb'}^{\tilde{\fa}\tilde{\fb'}}$. Furthermore, $a_{\fa\fb}^{\tilde{\fa}\tilde{\fb}}=0$ unless there exist $\fs,\ft\in\std(\bgl)$ such that $\bmu|\bnu(\fs)=\mathrm{S}$ and $\bga|\bgb(\ft)=\mathrm{T}$, which implies that $\fa=\fs_{\bmu\cap\bga}$, $\tilde{\fa}=\fs_{\bnu\cup\bgb}$, $\fb=\ft_{\bmu\cap\bga}$ and $\tilde{\fb}=\ft_{\bnu\cup\bgb}$. It completes the proof.
\end{proof}
Proposition~\ref{Prop:Basis} and Corollary~\ref{Cor:Hom} show that the following definition is well-defined.
\begin{definition} Suppose that $\bmu|\bnu,\bga|\bgb\in\mathcal{C}(\bk|\bl;n)$ and $\bgl\in H(\bk|\bl;n)$. For $\mathrm{S}\in \std(\bgl,\bmu|\bnu)$ and $\mathrm{T}\in \std(\bgl, \bga|\bgb)$, we define a homomorphism $\varphi_{\mathrm{ST}}\in \mathrm{Hom}_{H}(M_{\bga|\bgb}, M_{\bmu|\bnu})$ by
  \begin{equation*}
    \varphi_{\mathrm{ST}}(m_{\bga|\bgb}h)=m_{\mathrm{ST}}h
  \end{equation*}
  for all $h\in H$. Note that $\varphi_{\mathrm{ST}}$ is a homogeneous homomorphism with degree $|\bnu|+|\bgb|(\mathrm{mod}\,2)$ and we can extend $\varphi_{\mathrm{ST}}$ to an element of the cyclotomic $q$-Schur superalgebra $\mathscr{S}(\bk|\bl)$ by defining $\varphi_{\mathrm{ST}}(m_{\bkp|\btau})=0$ when $\bkp|\btau\neq \bga|\bgb$.
\end{definition}

\begin{definition}\label{Def:bar-varphi}
For $\bgl\in H(\bk|\bl;n)$, we define $\overline{\mathscr{S}(\bk|\bl;n)}_{\bgl}$ to be the $R$-submodule of $\mathscr{S}(\bk|\bl;n)$ spanned by \begin{equation*}
     \left\{\varphi_{\mathrm{ST}}\left|\begin{array}{l}\mathrm{S}\in \std(\tilde{\bgl},\bmu|\bnu), \mathrm{T}\in \std(\tilde{\bgl},\bga|\bgb) \text{ for some }\bmu|\bnu, \\\bga|\bgb\in \mathcal{C}(\bk|\bl;n) \text{ and } \tilde{\bgl}\vdash n \text{ with }\tilde{\bgl}\rhd \bgl\end{array} \right. \right\}.
   \end{equation*}
   \end{definition}

   Now we can state the main results of the paper.
\begin{theorem}\label{Them:main}
\begin{enumerate}
  \item[(i)]The cyclotomic $q$-Schur superalgebra $\mathscr{S}(\bk|\bl;n)$ is a free $R$-module with homogeneous basis
  \begin{equation*}
     \left\{\varphi_{\mathrm{ST}}\left|\begin{array}{l}\mathrm{S}\in \std(\bgl,\bmu|\bnu), \mathrm{T}\in \std(\bgl,\bga|\bgb) \text{ for some }\bmu|\bnu, \\\bga|\bgb\in \mathcal{C}(\bk|\bl;n)  \text{ and } \bgl\in H(\bk|\bl;n)\end{array}\right. \right\}.
   \end{equation*}
  \item[(ii)] Suppose that $\bmu|\bnu,\bga|\bgb\in \mathcal{C}(\bk|\bl;n)$ and $\bgl\in H(\bk|\bl;n)$ and let $\varphi\in \mathscr{S}(\bk|\bl;n)$. Then for every $\bga|\bgb\in\mathcal{C}(\bk|\bl;n)$ and $\mathrm{U}\in \std(\bgl,\bga|\bgb)$ there exists $r_{\mathrm{U}}\in R$ such that for all $\mathrm{S}\in \std(\bgl, \bmu|\bnu)$ and $\mathrm{T}\in \std(\bgl, \bga|\bgb)$,
      \begin{equation*}
       \varphi_{\mathrm{ST}}\varphi\equiv\sum_{\bkp|\btau\in \mathcal{C}(\bk|\bl;n) }\sum_{\mathrm{U}\in \std(\bgl,\bkp|\btau)}r_{\mathrm{U}}
       \varphi_{\mathrm{SU}}\mathrm{mod\,}\overline{\mathscr{S}(\bk|\bl;n)}_{\bgl}.
      \end{equation*}
\end{enumerate}
\end{theorem}
\begin{proof}
  (i) Corollary~\ref{Cor:Hom}(ii) implies that the morphism $\varphi\mapsto \varphi(m_{\bmu|\bnu})$ gives an isomorphism between $\mathrm{Hom}_{H}(M_{\bga|\bgb}, M_{\bmu|\bnu})$ and $M_{\bga|\bgb}^*\cap M_{\bmu|\bnu}$. Thus  Proposition~\ref{Prop:Basis} and Definition~\ref{Def:bar-varphi} shows the set $\{\varphi_{\mathrm{S}\mathrm{T}}\}$ is a basis of  $\mathscr{S}(\bk|\bl;n)$.

  (ii) We may assume that $\varphi\in \mathrm{Hom}_{H}(M_{\bkp|\btau}, M_{\bga|\bgb})$ for some $\bkp|\btau\in \mathcal{C}(\bk|\bl;n)$. Suppose that $\varphi(m_{\bkp|\btau})=m_{\bga|\bgb}h$ for some $h\in H$. Then,  for all $\mathrm{S}\in \std(\bgl, \bmu|\bnu)$ and $\mathrm{T}\in \std(\bgl, \bga|\bgb)$, we have
  \begin{equation*}
    \varphi_{\mathrm{ST}}\varphi(m_{\bkp|\btau})=m_{\mathrm{ST}}h\in M_{\bkp|\btau}^*\cap M_{\bmu|\bnu}.
  \end{equation*}
  By applying Proposition~\ref{Prop:Basis}, there exists $\tilde{\bgl}\in H(\bk|\bl;n)$,  $r_{\mathrm{UV}}\in R$ such that
  \begin{equation*}
    m_{\mathrm{ST}}h=\sum_{\mathrm{U}\in \std(\tilde{\bgl},\bmu|\bnu),\mathrm{V}\in \std(\tilde{\bgl},\bkp|\btau)}r_{\mathrm{UV}}m_{\mathrm{UV}}.
  \end{equation*}
Thank to Remark~\ref{Remark:m-ST}, \cite[Proposition~3.25]{DJM} shows
\begin{equation*}
  m_{\mathrm{ST}}h=\sum_{\widetilde{\mathrm{T}}\in \std(\bgl,\bkp|\btau)}r_{\widetilde{\mathrm{T}}}m_{\mathrm{S}\widetilde{\mathrm{T}}}+
  \sum r_{\mathrm{UV}}m_{\mathrm{UV}},
\end{equation*}
where $r_{\widetilde{\mathrm{T}}}$, $r_{\mathrm{UV}}\in R$ and the second sum is over $\mathrm{U}\in \std(\tilde{\bgl},\bmu|\bnu)$ and $\mathrm{V}\in \std(\tilde{\bgl},\bkp|\btau)$ for some $\tilde{\bgl}\in H(\bk|\bl;n)$ with $\tilde{\bgl}\rhd \bgl$. So
\begin{equation*}
 \varphi_{\mathrm{ST}}\varphi=\sum_{\widetilde{\mathrm{T}}\in \std(\bgl,\bkp|\btau)}r_{\widetilde{\mathrm{T}}}m_{\mathrm{S}\widetilde{\mathrm{T}}}
 \mathrm{mod\,}\overline{\mathscr{S}(\bk|\bl;n)}_{\bgl}.
\end{equation*}
It completes the proof.
\end{proof}
\begin{theorem}\label{Them:Cellular-algebra}
$\mathscr{S}(\bk|\bl;n)$ is a cellular algebra with (homogeneous) cellular basis
  \begin{equation*}
     \left\{\varphi_{\mathrm{ST}}\left|\begin{array}{l}\mathrm{S}\in \std(\bgl,\bmu|\bnu), \mathrm{T}\in \std(\bgl,\bga|\bgb) \text{ for some }\bmu|\bnu, \\\bga|\bgb\in \mathcal{C}(\bk|\bl;n)  \text{ and } \bgl\in H(\bk|\bl;n)\end{array}\right. \right\}.
   \end{equation*}
   and anti-involution $*$ defined by $\varphi_{\mathrm{ST}}^*=\varphi_{\mathrm{TS}}$.
\end{theorem}

\begin{proof}Recall that the the anti-automorphism $*$ of $H$ defined in \S2.4. Then
   $m_{\bmu|\bnu}^*=m_{\bmu|\bnu}$ for all $\bnu|\bmu\in \mathcal{C}(\bk|\bl;n)$. Furthermore, for $\mathrm{S}\in \std(\bgl,\bmu|\bnu)$ and $\mathrm{T}\in\std(\bgl, \bga|\bgb)$ for some $\bmu|\bnu,\bga|\bgb\in \mathcal{C}(\bk|\bl;n)$, $\bgl\in H(\bk|\bl;n)$, we have
   \begin{equation*}
     \varphi_{\mathrm{ST}}(m_{\bga|\bgb})=m_{\mathrm{ST}}=m_{\mathrm{TS}}^*= \varphi_{\mathrm{TS}}(m_{\bmu|\bnu})^*.
   \end{equation*}
    Therefore \cite[Lemma~6.10]{DJM} shows  $*$ is the anti-automorphism of $\mathscr{S}(\bk|\bl;n)$.

    Now Theorem~\ref{Them:main}(ii) shows that $\overline{\mathscr{S}(\bk|\bl;n)}_{\bgl}$ is a right ideal of $\mathscr{S}(\bk|\bl;n)$. By applying the anti-automorphism $*$, we deduce that it is also a left ideal of $\mathscr{S}(\bk|\bl;n)$. Hence the theorem follows by applying Theorem~\ref{Them:main}.
\end{proof}

\section{Applications}\label{Sec:Weyl-supermodules}
This section deals with the representation theory of $\mathscr{S}(\bk|\bl;n)$ by applying the theory of cellular algebras. We also establish the Schur-Weyl duality between the cyclotomic Hecke algebra and the cyclotomic $q$-Schur superalgebra at the end of the section.

 For $\bgl\in H(\bk|\bl;n)$, if we let $\mathrm{T}^{\bgl}=\bgl_{\sharp}|\bgl_{\ast}(\ft^{\bgl})$ (see Definition~\ref{Def:mu|nu(s)}), then $\mathrm{T}^{\bgl}$ is the unique $(\bk,\bl)$-semistandard $\bgl$-tableau of type $\bgl_{\sharp}|\bgl_{\ast}$ (see Example~\ref{Exam:tableau-type}). Define $\varphi_{\bgl}=\varphi_{\mathrm{T}^{\bgl}\mathrm{T}^{\bgl}}$, then $\varphi_{\bgl}$ is the identity map on $M_{\bgl_{\sharp}|\bgl_{\ast}}$.

\begin{definition}For $\bgl\in H(\bk|\bl;n)$, we define the \textit{Weyl supermodule} $W_{\bgl}$ to be the submodule of $\mathscr{S}(\bk|\bl;n)/\mathscr{S}(\bk|\bl;n)_{\bgl}$ generated by $\varphi_{\bgl}$.
\end{definition}

For a $(\bk,\bl)$-semistandard $\bgl$-tableau $\mathrm{S}$, we define
\begin{equation*}
  \varphi_{\mathrm{S}}:=\varphi_{\mathrm{ST}^{\bgl}}(\varphi_{\bgl}+
  \overline{\mathscr{S}(\bk|\bl;n)_{\bgl}})=\varphi_{\mathrm{ST}^{\bgl}}+
  \overline{\mathscr{S}(\bk|\bl;n)}_{\bgl}.
\end{equation*}
The following result follows directly by applying Theorem~\ref{Them:Cellular-algebra}.

\begin{corollary}
  The Weyl supermodule $W_{\bgl}$ is a free $R$-supermodule with homogeneous basis
  \begin{equation*}
    \{\varphi_{\mathrm{S}}|\mathrm{S}\in \std(\bgl,\bmu|\bnu)\text{ for some }\bmu|\bnu\in \mathcal{C}(\bk|\bl;n)\}.
  \end{equation*}
\end{corollary}

Now we define a bilinear form $\langle\,,\rangle$ on $W_{\bgl}$ by letting
\begin{equation*}
  \varphi_{\mathrm{T}^{\bgl}\mathrm{S}}\varphi_{\mathrm{T}\mathrm{T}^{\bgl}}\equiv \langle \varphi_{\mathrm{S}}, \varphi_{\mathrm{T}}\rangle\mod \overline{\mathscr{S}(\bk|\bl;n)_{\bgl}}
\end{equation*}
for all $(\bk,\bl)$-semistandard $\bgl$-tableaux $\mathrm{S}$ and $\mathrm{T}$. Then this bilinear form is well-defined and symmetric, and also satisfies $\langle \phi\varphi, \psi\rangle=\langle\varphi,\phi^*\psi\rangle$ for all $\varphi,\psi\in W_{\bgl}$ and $\phi\in\mathscr{S}(\bk|\bl;n)$ (see \cite[Proposition~2.4]{GL}). Consequently,
\begin{equation*}
  \mathrm{rad\,}W_{\bgl}=\{\phi\in W_{\bgl}|\langle \phi, \psi\rangle=0\text{ for all }\psi\in W^{\bgl}\}
\end{equation*}
is a submodule of $W_{\bgl}$. Now we define $F_{\bgl}=W_{\bgl}/\mathrm{rad\,}(W_{\bgl})$.
\begin{theorem}\label{Them:irr}Notations being as above and suppose that $R$ is a field. Then
\begin{equation*}
  \{F_{\bgl}|\bgl\in H(\bk|\bl;n)\}
\end{equation*}
is a complete set of non-isomorphic irreducible $\mathscr{S}(\bk|\bl;n)$-modules. Moreover, each $F_{\bgl}$ is absolutely irreducible.
\end{theorem}
\begin{proof}
  For $\bgl\in H(\bk|\bl;n)$, we have
  \begin{equation*}
    \varphi_{\mathrm{T}^{\bgl}\mathrm{T}^{\bgl}}\varphi_{\mathrm{T}^{\bgl}\mathrm{T}^{\bgl}}
    \equiv \langle \varphi_{\mathrm{T}^{\bgl}}, \varphi_{\mathrm{T}^{\bgl}}\rangle\,\mod\overline{\mathscr{S}(\bk|\bl;n)_{\bgl}}.
  \end{equation*}
  While $\varphi_{\mathrm{T}^{\bgl}\mathrm{T}^{\bgl}}\varphi_{\mathrm{T}^{\bgl}\mathrm{T}^{\bgl}}
  =\varphi_{\bgl}$ is the identity on $M_{\bgl_{\sharp}|\bgl_{\ast}}$, thus $\langle \varphi_{\mathrm{T}^{\bgl}}, \varphi_{\mathrm{T}^{\bgl}}\rangle=1$ and $F_{\bgl}\neq 0$. The theorem follows from \cite[Thereom~3.4]{GL}.
\end{proof}

\begin{remark}For $\bgl,\bmu\in H(\bk|\bl;n)$, we denote $d_{\bgl\bmu}$ the composition multiplicity of $F_{\bmu}$ as a composition factor of $W_{\bgl}$. Then $(d_{\bgl\bmu})_{\bgl,\bmu\in H(\bk|\bl;n)}$ is the \textit{decomposition matrix} of $\mathscr{S}(\bk|\bl;n)$. The theory of cellular algebras \cite[Proposition~3.6]{GL} shows the decomposition matrix of $\mathscr{S}(\bk|\bl;n)$ is unitriangular.
Finally, combining Theorem~\ref{Them:irr} and \cite[Remark~3.10]{GL}, we yield that the cyclotomic $q$-Schur superalgebra is quasi-hereditary.
\end{remark}

\begin{remark}By applying similar argument, we can show that the cyclotomic $q$-Schur superalgebra $\mathscr{S}(\bk|\bl;\Lambda)$ has the same properties as those of $\mathscr{S}(\bk|\bl;n)$ for any saturated subset $\Lambda$ of $\mathcal{C}(\bk|\bl;n)$.
\end{remark}

The following fact can be proved by applying a similar argument as that of \cite[Theorem~5.3]{Mathas-ASPM} (see also \cite[\S7]{SW}).
\begin{theorem}[Double centralizer property]\label{Them:Schur-Weyl}Notations being as above. Then
\begin{equation*}
  \mathscr{S}(\bk|\bl;n)\cong \mathrm{End}_{H}\biggl(\bigoplus_{\bmu|\bnu\in\mathcal{C}(\bk|\bl;n)}M_{\bmu|\bnu}\biggr)\text{ and } H\cong\mathrm{End}_{\mathscr{S}(\bk|\bl;n)}
  \biggl(\bigoplus_{\bmu|\bnu\in\mathcal{C}(\bk|\bl;n)}M_{\bmu|\bnu}\biggr).
\end{equation*}
\end{theorem}

We remark in closing that there are various cyclotomic $q$-Schur algebras. Let us mention some of them: Ariki \cite{A} introduced a cyclotomic $q$-Schur algebra as the endomorphism algebra of the action of $H$ on a $q$-tensor space, which is a quotient of the quantum algebra; Lin and Rui \cite{LR} defined a cyclotomic $q$-Schur algebra via the affine tensor space, which is a quotient of the affine quantum algebra;  Sawada and Shoji \cite{SS-Schur} defined the cyclotomic $q$-Schur algebra as the endomorphism algebra of the action of modified Ariki-Koike algebras on the $q$-tensor space; Deng et.\,al.\,\cite{DDY} introduce the slim cyclotomic $q$-Schur algebra.  It may be interesting to formulate the super-versions of these cyclotomic $q$-Schur algebras.
\bibliographystyle{amsplain}

\end{document}